\documentclass[12pt,reqno]{amsart}
\usepackage{amssymb,mathrsfs}
\usepackage{colortbl}

\def\jb#1{\langle#1\rangle}
\def\norm#1{\|#1\|}
\def\normo#1{\left\|#1\right\|}

\newcommand{\E}{\mathcal{E}}

\newcommand{\LL}{\mathcal{L}}
\newcommand{\M}{\mathcal{M}}
\newcommand{\cS}{\mathcal{S}}
\newcommand{\cH}{\mathcal{H}}

\newcommand{\K}{\mathcal{K}}

\newcommand{\C}{\mathbb{C}}
\newcommand{\N}{\mathbb{N}}
\newcommand{\R}{\mathbb{R}}

\newcommand{\al}{\alpha}

\newcommand{\fy}{\varphi}
\newcommand{\om}{\omega}

\newcommand{\De}{\Delta}

\newcommand{\na}{\nabla}

\newcommand{\les}{\lesssim}

\newcommand{\EQ}[1]{\begin{equation}\begin{split} #1 \end{split}\end{equation}}
\newcommand{\Del}[1]{}
\newcommand{\CAS}[1]{\begin{cases} #1 \end{cases}}

\newcommand{\pt}{&}

\numberwithin{equation}{section}

\newtheorem{thm}{Theorem}[section]
\newtheorem{cor}[thm]{Corollary}
\newtheorem{lem}[thm]{Lemma}
\newtheorem{prop}[thm]{Proposition}

\theoremstyle{remark}
\newtheorem{rem}{Remark}

\begin{document}
\subjclass[2010]{35L70, 35Q55}
\keywords{Nonlinear Schr\"odinger equation; combined nonlinearities, global wellposedness; scattering, blowup.}

\title[mass-critical combined Schr\"odinger equation]{Global well-posedness and scattering for nonlinear Schr\"odinger equations with combined nonlinearities in the radial case}

\author{Xing Cheng}
\address{Wu Wen-Tsun Key Laboratory of Mathematics and School of Mathematical Sciences, University of Science and Technology of China,
Hefei 230026,\  Anhui,\  China} \email{chengx@mail.ustc.edu.cn}

\author{Changxing Miao}
\address{Institute of Applied Physics and Computational Mathematics,
P. O. Box 8009,\ Beijing 100088,\ China} \email{miao\_changxing@iapcm.ac.cn}

\author{Lifeng Zhao}
\address{Wu Wen-Tsun Key Laboratory of Mathematics and School of Mathematical Sciences, University of Science and Technology of China, Hefei 230026,\  Anhui,\  China}\email{zhaolf@ustc.edu.cn}

\begin{abstract}We consider the Cauchy problem for the nonlinear
Schr\"odinger equation with combined nonlinearities, one of which is defocusing mass-critical and the other is focusing energy-critical or
 energy-subcritical. The threshold is given by means of variational argument. We establish the profile decomposition in $H^1(\R^d)$ and
  then utilize the concentration-compactness method to show the global wellposedness and scattering versus blowup in $H^1(\Bbb R^d)$
  below the threshold for radial data when $d\leq4$. 
\end{abstract}

\maketitle

\tableofcontents

\section{Introduction}
We will consider the Cauchy problem for the nonlinear Schr\"odinger equation
\begin{equation}\label{eq1.1}
\CAS {i\partial_t u  + \Delta u = |u|^\frac4d u - |u|^{p-1} u,\\
     u(0) = u_0 \in H^1(\R^d),}
\end{equation}
where $u: \R \times \R^d \to \C$ is a complex-valued function,
$1+\tfrac{4}{d}<p\leq 1+\tfrac{4}{d-2},\  d = 3,\, 4$ and $1 +
\tfrac4d < p < \infty,\  d = 1,\,2$.

The equation \eqref{eq1.1} has the following mass and energy:
\begin{equation}\label{eq1.2}
  \M(u) = \int_{\R^d} |u|^2\,\mathrm{d}x,
\end{equation}
\begin{equation}\label{eq1.3}
   \E(u) = \int_{\R^d} \Big(\tfrac12 |\nabla u|^2 - \tfrac1{p+1} |u|^{p+1} + \tfrac{d}{2(d+2)} |u|^\frac{2(d+2)}d\Big) \,\mathrm{d}x.
\end{equation}
The equation (\ref{eq1.1}) is a special case of the general nonlinear Schr\"odinger equation with combined nonlinearities
\begin{equation}\label{gmc}
\CAS {i\partial_t u  + \Delta u = \mu_1|u|^{p_1-1} u + \mu_2|u|^{p_2-1} u,\\
     u(0) = u_0 \in H^1(\R^d),}
 \end{equation}
 where $1 < p_1,p_2\leq 1+\frac{4}{d-2},\, \rm{for}\  d\ge 3$, $\ 1< p_1, p_2 < \infty,\,  \rm{for}\ d = 1,\,2$, $\mu_1,\mu_2\in\{\pm1\}$.
 This equation arises in the study of the Hartree approximation and quasi-classically within the frame of the secondary quantization
 in the boson gas with many body $\delta$-function interaction. It can also be used to describe the effect of saturation of nonlinear refractive index.
  At the same time, the equation of nuclear hydrodynamics with effective Skyrme's forces reduces quasi-classically to \eqref{gmc}. For more physical background, we refer the reader to \cite{Barashenkov-Gocheva-Makhankov-Puzynin, Pelinovsky-Afanasjev-Kivshar} and the references therein.

The local theory for \eqref{gmc} in the energy space follows from the standard method of T. Cazenave, F. B. Weissler \cite{Cazenave-Weissler}, see also \cite{Cazenave}.
Proceeded by \cite{X-Zhang}, T. Tao, M. Visan and X. Zhang considered various cases in \cite{Tao-Visan-Zhang}. They proved global wellposedness and scattering of the solution to the equation (\ref{gmc}) for finite energy data when $\mu_1=\mu_2=+1$ and $1 + \frac{4}{d}\leq p_1<p_2  \le 1 +  \frac{4}{d-2},\  d \ge 3$. The case when $p_1= 1 + \frac{4}{d}$ and $p_2= 1 + \frac{4}{d-2}$ is the most difficult. In \cite{Tao-Visan-Zhang}, the low frequencies of the solution are well approximated by the $L^2$ critical problem and the high frequencies are well approximated by the energy-critical problem. The medium  frequencies are controlled by the Morawetz estimates.

In \cite{Miao-Xu-Zhao1}, C. Miao, G. Xu and L. Zhao considered the
case where $\mu_1=+1$, $\mu_2=-1$ and $1 + \tfrac{4}{d}<p_1<p_2 = 1
+ \tfrac{4}{d-2}$, $d = 3$. The threshold was given by variational
method due to the energy trapping property as in
\cite{Ibrahim-Masmoudi-Nakanishi}. They established the linear profile decomposition in $H^1(\R^d)$ in the spirt of \cite{Ibrahim-Masmoudi-Nakanishi}. By using this new profile decomposition, they reduced the scattering problem to the extinction of the critical element. The critical element can then be excluded
by using the Virial identity. They showed the dichotomy of
global wellposedness and scattering versus blow-up phenomenon below
the threshold for radial solutions.  The radial assumption was
removed in dimensions five and higher  in \cite{Miao-Xu-Zhao2}. While, for the case of
lower dimensions $d\in\{3,4\}$, how to remove the radial
assumption is still open in this field.

If $\mu_1=\mu_2=-1$ and $ 1 + \tfrac{4}{d}<p_1<p_2 = 1 +
\tfrac{4}{d-2}$, $d\ge 5$, the Cauchy problem was considered in
\cite{Akahori-Ibrahim-Kikuchi-Nawa1, Akahori-Ibrahim-Kikuchi-Nawa2}. After giving existence of the ground state based on the
idea in \cite{Brezis-Nirenberg} and
\cite{Ibrahim-Masmoudi-Nakanishi}, they showed a sufficient and
necessary condition for the scattering in the spirit of
\cite{Kenig-Merle} by using the profile decomposition in $\dot{H}^1(\R^d)$ and the global wellposedness and scattering result in \cite{Killip-Visan}.

In this paper, we aim to look for the suitable threshold to study the global well-posedness and scattering versus blowup of \eqref{eq1.1}.
 Before stating the main theorem, we introduce some notations. For $\varphi \in H^1(\R^d)$, we denote the scaling
$(T_\lambda \varphi)(x) = \lambda^\frac{d}2 \varphi(\lambda x)$.
For any $ \omega > 0$, we have the Lyapunov functional 
\begin{equation}\label{eq1.4}
\cS_\omega (\varphi) = \E(\varphi) + \tfrac12\omega \M(\varphi).
\end{equation}
We also denote the scaling derivative of $\cS_\omega (\varphi)$ by $\K(\varphi)$,
\EQ{ \label{eq1.6}
\K(\varphi) \pt=  \LL\,  \cS_\omega(\varphi) = \tfrac{\mathrm{d}}{\mathrm{d}\lambda}\big|_{\lambda= 1} \cS_\omega (T_\lambda \varphi)
=\tfrac{\mathrm{d}}{\mathrm{d}\lambda}\big|_{\lambda = 1} \E(T_\lambda \varphi)\\
               \pt    = \int_{\R^d} |\nabla \varphi|^2 - \tfrac{d(p-1)}{2(p+1)} |\varphi|^{p+1} + \tfrac{d}{d+2} |\varphi|^\frac{2(d+2)}d \,\mathrm{d}x.}
Let
\begin{equation}\label{eq1.7}
  m_\omega = \inf\{ \cS_\omega(\varphi): \varphi \in H^1(\R^d)\setminus \{0\}, \K (\varphi) = 0\},
\end{equation}
then we have
\begin{prop}\label{bcpr1.1}
For $1 + \frac4d < p < 1 + \frac4{d-2},\, d\ge 3, \text{ and } \ 1 + \frac4d < p < \infty,\, d= 1,2$, $\om > 0$,
 we have $m_\om > 0$. Moreover, $m_\om = \cS_\om(Q)$,
where $Q\in H^1(\R^d)$ is the ground state of
\begin{equation} \label{eq2.8}
- \om Q + \De Q + |Q|^{p-1} Q - |Q|^\frac4d Q = 0.
\end{equation}
\end{prop}
\begin{prop}\label{pr1.1}
For $p = 1 + \frac4{d-2},\, d \ge 3$, we have for all $\omega > 0$,
\begin{equation}\label{eq2.14v2}
m_\omega = \E^0(W),
\end{equation}
where
\begin{equation*}
\E^0(W) = \int_{\R^d} \Big(\tfrac12|\nabla W|^2
-\tfrac{d-2}{2d} |W|^\frac{2d}{d-2}\Big)\,\mathrm{d}x,
\end{equation*}
and $W\in \dot{H}^1(\R^d)$ is unique positive radial solution of
\begin{equation} \label{eq2.16v2}
- \Delta W = |W|^\frac4{d-2} W.
\end{equation}
\end{prop}
\begin{rem}
\begin{enumerate}
 \item When $p = 1 + \frac4{d-2}$, we note $m_\omega = \E^0(W)$, which indicates that $m_\omega$ is independent of $\omega$, so we can define $m :=  m_\omega =\E^0(W)$ in this case. Let
\begin{align*}
& A_{\omega, +} = \{ \varphi\in H^1(\R^d): \cS_\omega(\varphi) < m_\omega,\, \K(\varphi) \ge 0\},\\
& A_{\omega, -} = \{ \varphi\in H^1(\R^d): \cS_\omega(\varphi) < m_\omega,\, \K(\varphi) < 0\},\\
& A_+ = \{ \varphi\in H^1(\R^d): \E(\varphi) < m,\, \K(\varphi) \ge 0\},\\
& A_- = \{ \varphi\in H^1(\R^d): \E(\varphi) < m,\, \K(\varphi) < 0\}.
\end{align*}
\item $A_{\omega, \pm}$ and $A_{\pm}$ are non-empty. In fact, we note $\varphi = 0$ belongs to both $A_{\omega, +}$ and $A_{+}$, so $A_{\omega, +}$ and $A_{+}$ are nonempty. On the other hand, we can easily verify that $\varphi(x) = \epsilon^{-\frac{d}2} Q(\epsilon^{-1}x)$ belongs to $A_{\omega,-}$, when $\epsilon$ is sufficiently smalll. Similarly, by using some truncation of $\varphi = \epsilon^{-\frac{d}2} W(\epsilon^{-1} x)$ (still denoted by $\varphi$) to make sure $\varphi\in H^1(\R^d)$, we can show $A_-$ is nonempty.
\item We will give the proof of the following theorem in a unified form regardless of $1+\tfrac4d < p < 1 + \tfrac4{d-2},\, d\in\{3,4\}$ or $1 + \tfrac4d < p
< \infty, \, d\in\{ 1, 2\}$ or $p = 1 + \frac4{d-2}, \, d =  3, 4$. In fact, We note for any $ u_0 \in A_+$, we have $\E(u_0) < m$, we can take $\omega > 0$ small enough such that $\E(u_0) + \frac12 \omega \M(u_0) < m$, thus $u_0 \in A_{\omega, +}$. Similarly, for any $u_0\in A_-$, we can still take $\omega > 0$ small enough such that $u_0\in A_{\omega, -}$. Thus, our result
    below in the situation $p = 1 + \frac4{d-2}, \, d =  3, 4$ can be proved as in the case $1+\tfrac4d < p < 1 + \tfrac4{d-2},\, d\in\{3,4\}$. 
\end{enumerate}
\end{rem}
Now, we state our result.
\begin{thm}\label{th1.4}
\begin{enumerate}
  \item For $1+\tfrac4d < p < 1 + \tfrac4{d-2},\, d\in\{3,4\}$, $1 + \tfrac4d < p
< \infty, \, d\in\{ 1, 2\}$, and $u_0$ is radial, we have for any $\omega > 0$, \\
 $(i)$ if $u_0
\in A_{\om,+}$, the solution $u$ to
$\eqref{eq1.1}$ exists globally and scatters in $H^1(\R^d)$;\\
 $(ii)$
if $u_0 \in A_{\om,-}$, for $d\ge 2, \ p \le
\min(5, 1+ \frac4{d-2})$, the solution $u$ blows up in finite time.
\item For $p = 1 + \frac4{d-2},\, d\in\{3,4\}$, and $u_0$ is radial, we have\\
$(i)$ if $u_0\in A_{ +}$, the solution $u$ to
\eqref{eq1.1} exists globally and scatters in $H^1(\R^d)$;\\
$(ii)$ if $u_0 \in A_{-}$, the solution $u$
blows up in finite time.
\end{enumerate}

\end{thm}
For the nonlinear Schr\"odinger equation with combined nonlinearities, we focus on the different roles played by the two nonlinearities. Generally speaking, the main barrier of the local theory is the higher order term while the lower order term is dominant in the global behavior.
 We prove our main theorem by the compactness-contradiction method initiated by C. E. Kenig and F. Merle \cite{Kenig-Merle}. In the argument,
 the linear profile decomposition plays an important role. In previous works such as \cite{Miao-Xu-Zhao1, Miao-Xu-Zhao2}, the authors considered the equation
 \begin{equation}\label{xmz}
\CAS {i\partial_t u  + \Delta u = |u|^{p-1} u - |u|^\frac{4}{d-2} u,\\
     u(0) = u_0 \in H^1(\R^d),}
 \end{equation}
where $1+\frac{4}{d}<p<1+\frac{4}{d-2},\, d\ge 3$. Since the lower order term is $\dot{H}^s(\R^d)$($0<s=\frac{d}{2}-\frac{2}{p-1}<1$) critical, the solution of (1.11) is expected to behave like that of the defocusing equation $i\partial_t u + \Delta u = |u|^{p-1} u$. The critical space of this defocusing equation is $\dot{H}^s(0 < s < 1)$, so it is reasonable to apply the
$\dot{H}^s$-profile decomposition to $\frac{\jb{\nabla}}{|\nabla|^s} u$ for $u\in H^1(\R^d)$.
Since the symmetry group in $\dot{H}^s$ is of the same type as in $\dot{H}^1$ case,
 we may equivalently apply the $\dot{H}^1$-profile decomposition to $\frac{\jb{\nabla}}{|\nabla|} u$ for $u\in H^1(\R^d)$.
However, the lower term in (\ref{eq1.1}) is $L^2(\R^d)$-critical. The symmetry group in $L^2(\R^d)$ is different from that in $\dot{H}^1(\R^d)$ due to the Galilean symmetry. Therefore, it is natural that the $L^2$-profile decomposition is applied to $\langle\nabla\rangle u$ for $u\in H^1(\R^d)$ in this paper.

Although the Galilean transform and scaling are encoded in the profile decomposition in $L^2$, they turn out to be excluded in the profile decomposition in $H^1$. In fact, for a typical profile, if the scaling parameter goes to zero then the $\dot{H}^1$ boundedness is violated. If the scaling parameter goes to infinity, then we can show such a profile vanishes. Similar case occurs for the Galilean transform.  As a consequence of the linear profile decomposition, we can reduce to almost periodic solution independent of wellposedness and scattering results of energy-critical or mass-critical equations. This is a striking difference from \cite{Akahori-Ibrahim-Kikuchi-Nawa2, Miao-Xu-Zhao1}, where scaling to zero was excluded by the non-existence of the almost periodic solution to the focusing energy-critical Schr\"odinger equations. Hence it relies  heavily on the results in \cite{Kenig-Merle, Killip-Visan}.

We explore the profile decomposition in use to get better estimate \eqref{eq2.15v3new} for the remainder. This estimate provides spacetime control for both the remainder $w_n^k$  and its derivative $|\nabla|w_n^k$. This makes remarkable difference with the profile decomposition for $H^1(\R^d)$ data obtained both in \cite{Carles-Keraani} and \cite{Miao-Xu-Zhao1}. In fact, R. Carles and S. Keraani \cite{Carles-Keraani} only provided the control over the remainder, while the authors in \cite{ Miao-Xu-Zhao1} only provided the control over the derivative of the remainder.

Although we  make more delicate analysis due to the stronger control in the perturbation theorem, we  get stronger compactness for the critical element in $H^1(\R^d)$, which is stronger than the compactness in $\dot{H}^s(0<s\le 1)$ obtained in \cite{Miao-Xu-Zhao1, Miao-Xu-Zhao2}.

We expect our result will be extended to higher dimensions($d\ge 5$) since all the arguments make sense except the long-time
perturbation. However, the exotic Strichartz estimates in \cite{Foschi, Vilela} seem useless to establish the long-time
perturbation in our case because the mass-critical term in our equation cannot be controlled properly in the Sobolev spaces.
The radial assumption is expected to be removed in a forthcoming paper.
%

The rest of the paper is organized as follows. After introducing some notations and preliminaries, we give the threshold in Section 2. Moreover, we show the energy-trapping properties for the set $A_{\om, \pm}$ in this section. The local wellposedness and perturbation theory are stated in Section 3. In Section 4, we derive the linear profile decomposition for data in $H^1(\R^d)$. Then we argue by contradiction. We reduce to the existence of a critical element in Section 5 and show the extinction of such a critical element in Section 6. To make the results complete, we show the existence of blowup solutions in Section 7.

\vskip 0.2in

{\bf Notation and Preliminaries}
We will use the notation $X\lesssim Y$ whenever there exists some positive constant $C$ so that $X\le C Y$. Similarly, we will use $X\sim Y$ if $X\lesssim Y \lesssim X$.

    We define the Fourier transform on $\R^d$ to be $$\hat{f}(\xi) =\tfrac1{(2\pi)^d} \int_{\R^d} e^{- ix\xi} f(x)\,\mathrm{d}x,$$ and for $s \in \R$,
the fractional differential operators $|\nabla|^s$ is defined by $\widehat{|\nabla|^s f}(\xi)  = |\xi|^s \hat{f}(\xi).$
  We also define
 $\jb{\nabla}^s$ by $\widehat{\jb{\nabla}^s f}(\xi)= ( 1 + |\xi|^2)^\frac{s}2 \hat{f}(\xi)$.

 We define the homogeneous Sobolev norms
 \[ \normo{f}_{\dot{H}^s(\R^d)} = \normo{|\nabla|^s f}_{L^2(\R^d)},\]
 and inhomogeneous Sobolev norms
 \[  \normo{f}_{H^s(\R^d)} = \normo{\jb{\nabla}^s f}_{L^2(\R^d)}.\]
  We use the notation $o_n(1)$ to denote a quantity which tends to 0, as $n \to \infty$.

For $I \subset \R$, we use $L_t^q L_x^r(I \times \R^d)$ to denote the spacetime norm
$$\|u\|_{L_t^q L_x^r(I \times \R^d)} =\bigg(\int_I \Big(\int_{\R^d} |u(t,x)|^r \,\mathrm{d}x\Big)^\frac{q}{r}\,\mathrm{d}t\bigg)^\frac1 q.$$
When $q = r$, we abbreviate $L_t^q L_x^r$ as $L_{t,x}^q$.

We also recall \emph{Duhamel's formula}
\begin{align}\label{duhamel}
u(t) = e^{i(t-t_0)\Delta}u(t_0) - i \int_{t_0}^t e^{i(t-s)\Delta}(iu_t + \Delta u)(s) ds.
\end{align}
  We say that a pair of exponents $(q,r)$ is $L^2$-\emph{admissible} if $\tfrac{2}{q} +
\tfrac{d}{r} = \frac{d}{2}$ and $2 \leq q,r \leq \infty, \ (q,r,d) \ne (2, \infty, 2), \ d\ge 1$.

\begin{lem}[Strichartz estimate, \cite{Keel-Tao}]\label{le2.1}
   Let I be a compact time interval and let $u: I\times \R^d \to \C$ be a solution to the forced Schr\"odinger equaton
$i\partial_t u + \Delta u = G$ for some function $G$, then we have
 \[ \| u\|_{L_t^q L_x^r(I \times \R^d)} \lesssim \|u(t_0)\|_{L_x^2(\R^d)} + \|  G\|_{L_t^{\tilde{q}'} L_x^{\tilde{r}'}(I\times \R^d)}\]
for any $t_0 \in I$ and any $L^2$-admissible exponents $(q,r)$, $(\tilde{q},\tilde{r})$.
\end{lem}
If $I \times \R^d$ is a spacetime slab, we
define the \emph{Strichartz norm} ${S}^0(I)$ by
\begin{equation}
\|u\|_{{S}^0(I) } = \sup \|u \|_{L_t^q L_x^r(I\times \R^d)},
\end{equation}
where the $\sup$ is taken over all $L^2$-admissible pairs $(q,r)$.
When $d = 2$, we need to modify the norm a little, where the $\sup$ is taken over all $L^2$-admissible pairs with $q \ge 2 + \epsilon$,
 for $\epsilon > 0$ arbitrary small.
We also define $S^1(I)$ norm by
 \begin{equation}
\|u\|_{{S}^1(I) } = \|\jb{\nabla} u \|_{S^0(I )}.
\end{equation}
\section{Variational estimates}
In this section, we prove Proposition \ref{bcpr1.1} and \ref{pr1.1}. We show the existence of the ground state
together with the energy-trapping property for $A_{\omega, \pm}$, which will be used to show the scattering and blow-up.

Let $\K(\fy) = \K^Q(\fy) + \K^N(\fy)$, where
\begin{equation*}
\K^Q(\fy) = \int_{\R^d} |\na \fy|^2\,\mathrm{d}x,
\end{equation*}
\begin{equation*}
\K^N(\fy) = - \tfrac{d(p-1)}{2(p+1)}\int  |\fy|^{p+1}\,\mathrm{d}x +
\tfrac{d}{d+2}\int |\fy|^\frac{2(d+2)}d \,\mathrm{d}x.
\end{equation*}
We have the following basic fact about $\K(T_\lambda \varphi)$:
\begin{lem}[Dynamic behavior of $\K$ under the scaling]\label{le2.1}
For any $\varphi\in H^1(\R^d)\setminus \{0\}$, there exists a unique $\lambda_0(\varphi) > 0$ that
\begin{equation}\label{eq2.1}
\K(T_\lambda \varphi)
\begin{cases}
> 0,  &  0 < \lambda < \lambda_0(\varphi),\\
= 0,  & \lambda = \lambda_0(\varphi),\\
< 0,   & \lambda > \lambda_0(\varphi).
\end{cases}
\end{equation}
\end{lem}
\begin{proof}
An easy computation gives
\begin{equation*}
\tfrac1{\lambda^2} \K(T_\lambda \varphi) = \int |\nabla \varphi|^2 +
\tfrac{d}{d+2} |\varphi|^\frac{2(d+2)}d - \tfrac{d(p-1)}{2(p+1)}
\lambda^{\frac{d}2(p-1)-2} |\varphi|^{p+1} \,\mathrm{d}x,
\end{equation*}
which implies $\K(T_\lambda \varphi)>0$ for $\lambda>0$ sufficiently small. By
\begin{align*}
 \tfrac{\mathrm{d}}{\mathrm{d}\lambda}\left(\tfrac1{\lambda^2} \K(T_\lambda \varphi)\right)
= - \tfrac{d(p-1)}{2(p+1)} \left(\tfrac{d}2(p-1) -2\right)
\lambda^{\frac{d}2(p-1)-3} \int |\varphi|^{p+1} \,\mathrm{d}x < 0,
\end{align*}
we see $\tfrac1{\lambda^2} \K(T_\lambda \varphi)$ is monotone
decreasing with respect to $\lambda>0$.

Since
\begin{align*}
\K(T_\lambda \varphi) & = \lambda^2 \left( \|\nabla \varphi\|_{L^2}^2 + \tfrac{d}{d+2}\|\varphi\|_{L^{\frac{2(d+2)}d}}^\frac{2(d+2)}d\right)
- \lambda^{\frac{d}2(p-1)} \tfrac{d(p-1)}{2(p+1)} \|\varphi\|_{L^{p+1}}^{p+1} \\
& \to -\infty, \text{ as } \lambda \to \infty, 
\end{align*}
there exists a unique $\lambda_0 > 0$ such that $\K(T_{\lambda_0} \varphi) = 0$ and \eqref{eq2.1} follows.
\end{proof}
Now, we show the positivity of $\K$ near 0 in the energy space.
\begin{lem} \label{le2.2}
  For any bounded sequence $\fy_n \in H^1(\R^d)\setminus \{0\}$ with $\K^Q(\fy_n) \to 0$, as $n \to \infty$,
then for $n$ large enough, we have $\K(\fy_n) > 0$.
\end{lem}
\begin{proof}
By assumption,
$\|\nabla \varphi_n\|_{L^2} \to 0$, as $n\to \infty$.
Due to the interpolation and Sobolev inequalities,
we have
\begin{align*}
\|\varphi_n\|_{L^{p+1}}^{p+1} & \lesssim  \|\varphi_n\|_{L^2}^{p+1 - \frac{d}2(p-1)} \|\nabla \varphi_n\|_{L^2}^{\frac{d}2(p-1)}.\\
\end{align*}
Since $\frac{d}2(p-1) > 2$, for $n$ large enough, we see
\begin{align*}
\K(\varphi_n) & = \int_{\R^d} |\nabla \varphi_n|^2 - \tfrac{d(p-1)}{2(p+1)} |\varphi_n|^{p+1} + \tfrac{d}{d+2} |\varphi_n|^\frac{2(d+2)}d \,\mathrm{d}x\\
              & \ge \int_{\R^d} |\nabla \varphi_n|^2 \,\mathrm{d}x - o\big(\int_{\R^d} |\nabla \varphi_n|^2\,\mathrm{d}x\big)\\
             &  \sim \int_{\R^d} |\nabla \varphi_n|^2 \,\mathrm{d}x > 0.
\end{align*}
\end{proof}
\begin{lem}\label{le2.3}
For any $\fy \in H^1(\R^d)$, we have
\begin{equation}\label{eq2.2}
(2-\LL) \cS_\om(\fy) = \om \|\fy\|_{L^2}^2 +
\tfrac{d(p-1)-4}{2(p+1)} \|\fy\|_{L^{p+1}}^{p+1},
\end{equation}
\begin{equation}\label{eq2.3}
\LL(2 - \LL) \cS_\om(\fy)  = \tfrac{d(p-1)(d(p-1)-4)}{4(p+1)}
\|\fy\|_{L^{p+1}}^{p+1}.
\end{equation}
\end{lem}
\begin{proof}Direct computation shows that
\begin{align*}
\LL \|\nabla \varphi\|_{L^2}^2 = 2 \|\nabla \varphi\|_{L^2}^2,\\
\LL \|\varphi \|_{L^{p+1}}^{p+1} = \tfrac{d}2(p-1) \|\varphi\|_{L^{p+1}}^{p+1},\\
\LL \|\varphi\|_{L^\frac{2(d+2)}d}^\frac{2(d+2)}d = 2\|\varphi\|_{L^\frac{2(d+2)}d}^\frac{2(d+2)}d,
\end{align*}
hence we have
\begin{align*}
(2- \LL) \cS_\omega(\varphi)
& = 2 \cS_\omega(\varphi) - \K(\varphi)\\
& = \omega \|\varphi\|_{L^2}^2 + \tfrac{d(p-1)-4}{2(p+1)}\int_{\R^d}
|\varphi|^{p+1} \,\mathrm{d}x,
\end{align*}
\begin{align*}
\LL(2- \LL) \cS_\omega(\varphi)
& = \omega \LL\|\varphi\|_{L^2}^2 + \tfrac{d(p-1)-4}{2(p+1)} \LL \|\varphi\|_{L^{p+1}}^{p+1}\\
& = \tfrac{d(p-1)-4}{2(p+1)} \LL \|\varphi\|_{L^{p+1}}^{p+1}\\
& = \tfrac{(d(p-1)-4)d(p-1)}{4(p+1)} \|\varphi\|_{L^{p+1}}^{p+1}.
\end{align*}
\end{proof}
Due to the lack of positivity of $\cS_\omega(\varphi)$, we introduce a non-negative functional
\EQ{ \label{eq2.3new}
\cH_\om(\fy) \pt  = \big(1 - \tfrac{\LL}2\big) \cS_\om(\fy) \\
                 \pt = \cS_\om(\fy) - \tfrac12 \K(\fy)\\
                \pt  = \tfrac{\om}2 \|\fy\|_{L^2}^2 + \tfrac{d(p-1)-4}{4(p+1)} \|\fy\|_{L^{p+1}}^{p+1},}
then for any $ \fy \in H^1(\R^d)\setminus \{0\}$, we have $\cH_\om(\fy) \ge 0,\  \LL \cH_\om(\fy) \ge 0.$
\begin{prop}\label{le2.4}
\EQ{ \label{eq2.4}
m_\omega \pt = \inf \{ \cH_\om(\fy): \fy \in H^1\setminus \{0\}, \K(\fy) \le 0\}\\
          \pt  = \inf \{ \cH_\om(\fy): \fy \in H^1\setminus \{0\}, \K(\fy) < 0\}.}
\end{prop}
\begin{proof}
Since $\cS_\omega(\varphi) = \cH_\omega(\varphi)$ when $\K(\varphi) = 0$,
\begin{equation*}
m_\omega  \ge \inf\{ \cH_\omega(\varphi): \varphi \in H^1\setminus \{0\}, \K(\varphi) \le 0\}.
\end{equation*}
Claim: \begin{equation*}
m_\om
\le \inf\{ \cH_\omega(\varphi): \varphi \in H^1\setminus\{0\}, \K(\varphi) < 0\}.
\end{equation*}
In fact, $\forall\, \varphi \in H^1\setminus \{0\}$ with $\K(\varphi) < 0$, by Lemma \ref{le2.1}, there exists $0 < \lambda_0 < 1$ that $\K(T_{\lambda_0} \varphi) = 0$, then the fact $\LL \cH_\omega \ge 0$ implies
$\cS_\omega(T_{\lambda_0} \varphi) = \cH_\omega(T_{\lambda_0} \varphi) \le \cH_\omega(\varphi)$.

It suffices to show
\EQ{ \label{eq2.7v2}
\inf\{ \cH_\omega(\varphi): \varphi\in H^1\setminus\{0\}, \K(\varphi) \le 0\}
\ge \inf\{ \cH_\omega(\varphi): \varphi\in H^1\setminus\{0\}, \K(\varphi) < 0\}.
}
In fact, for any $\varphi \in H^1\setminus\{0\}$ with $\K(\varphi) \le 0$, by \eqref{eq2.3}, we know that
\begin{equation}\label{eq2.7v10}
\LL \,\K(\varphi)  = 2 \K(\varphi) -
\tfrac{d(p-1)(d(p-1)-4)}{4(p+1)} \|\varphi\|_{L^{p+1}}^{p+1} < 0,
\end{equation}
then for any $\lambda > 1$, we have $\K(T_\lambda \varphi) < 0$,
and as $\lambda \to 1$,
\begin{align*}
\cH_\omega(T_\lambda \varphi) & = \tfrac{\omega}2 \|T_\lambda \varphi\|_{L^2}^2 + \tfrac{d(p-1)-4}{4(p+1)} \|T_\lambda \varphi\|_{L^{p+1}}^{p+1}\\
  &   = \tfrac{\omega}2 \|\varphi\|_{L^2}^2 + \tfrac{d(p-1)-4}{4(p+1)} \lambda^{\frac{d}2(p-1)} \|\varphi\|_{L^{p+1}}^{p+1}\\
   & \to \tfrac{\omega}2 \|\varphi\|_{L^2}^2 + \tfrac{d(p-1)-4}{4(p+1)} \|\varphi\|_{L^{p+1}}^{p+1} = \cH_\omega(\varphi).
\end{align*}
This shows \eqref{eq2.7v2} and completes the proof.
\end{proof}
 We now give the value of $m_\omega$ for $1 + \tfrac4d < p < 1 + \tfrac4{d-2}, \, d\ge 3, \text{ and } 1 + \tfrac4d < p < \infty, d= 1,2$, namely prove Proposition \ref{bcpr1.1}.

{\it Proof of Proposition \ref{bcpr1.1}.}
Let $\varphi_n \in H^1(\R^d)$ be a minimizing sequence for \eqref{eq2.4}, namely
\[ \K(\varphi_n) \le 0, \ \varphi_n \ne 0, \ \cH_\omega(\varphi_n) \searrow m_\omega, \text{ as } n\to \infty.\]
Let $\varphi_n^*$ be the Schwartz symmetrization of $\varphi_n$, i.e. the radial decreasing rearrangement.
Since the symmetrization preserves the nonlinear parts and does not increase the $\dot{H}^1$ part, we have
$$\varphi_n^* \ne 0,\  \K(\varphi_n^*) \le \K(\varphi_n) \le 0 \text{ and  } \cH_\omega(\varphi_n^*) = \cH_\omega(\varphi_n) \to m_\omega, \text{ as } n\to \infty.$$
Then by Lemma \ref{le2.1} and \eqref{eq2.7v10}, there exists $0< \lambda_n \le 1$ such that $\psi_n = T_{\lambda_n} \varphi_n^*$ satisfies
\[ \psi_n \ne 0, \ \K(\psi_n) = 0,\ \cS_\omega(\psi_n) = \cH_\omega(\psi_n) \to m_\omega,\text{  as } n\to \infty.\]
Moreover, direct computation gives
\begin{equation*}
\tfrac{\omega}2 \|\psi_n\|_{L^2}^2 + \tfrac12 \|\nabla
\psi_n\|_{L^2}^2 \le \cS_\omega(\psi_n) + \tfrac4{d(p-1)-4}
\cH_\omega(\psi_n),
\end{equation*}
which implies the boundedness of $\psi_n$ in $H^1(\mathbb{R}^d)$.

 Then, $\psi_n$ converges weakly to some $\psi$ in $H^1(\R^d)$, up to a subsequence. Since $\psi_n$ is radial, it also converges strongly in
  $L^q(\R^d)$ for $2<  q < \tfrac{2d}{d-2}, d \ge 3 \text{ and }  2 < q < \infty, d = 1, 2$.
Thus, $\K(\psi) \le \underset{n\to \infty} \liminf \,\K(\psi_n) = 0$, $\cH_\omega(\psi) \le \underset{n\to \infty} \liminf\, \cH_\omega(\psi_n) =
  m_\omega$.
Moreover, $\psi \ne 0$. In fact, if $\psi = 0$, then $\K(\psi_n) = 0$ implies $K^Q(\psi_n) = - K^N(\psi_n) \to 0$, as $n\to \infty$, and
 by Lemma \ref{le2.2},
we have $\K(\psi_n) > 0$ for $n$ large, a contradiction.
Since $\K(\psi) \le 0$ and $\psi \ne 0$, we have $\cH_\omega(\psi) \ge  m_\omega$, so we have $\cH_\omega(\psi) =  m_\omega$ and $\K(\psi) \le 0$.

By scaling, we may replace $\psi$ by its rescaling, so that
\[ \K(\psi) = 0,\  \cS_\omega(\psi) = \cH_\omega(\psi) \le m_\omega \text{ and } \psi \ne 0.\]
Then $\psi$ is a minimizer and $m_\omega = \cH_\omega(\psi) > 0$.

By variational theory, there is a Lagrange multiplier $\eta \in \R$
such that $\cS_\omega'(\psi) = \eta \K'(\psi)$. We note
\begin{align*}
0 = \K(\psi) = \LL\cS_\omega(\psi) = \cS_\omega'(\psi) \LL\psi = \eta \K'(\psi) \LL\psi = \eta \cdot \LL^2 \cS_\omega(\psi).
\end{align*}
By \eqref{eq2.3} and $\LL \cS_\omega(\psi) = 0$, we have
\begin{align*}
\LL^2 \cS_\omega (\psi) & = 2 \LL \cS_\omega(\psi) - \tfrac{d(p-1)(d(p-1)-4)}{4(p+1)} \|\psi\|_{L^{p+1}}^{p+1}\\
                         & = - \tfrac{d(p-1)(d(p-1)-4)}{4(p+1)} \|\psi\|_{L^{p+1}}^{p+1} < 0,
\end{align*}
therefore $\eta = 0$ and $\psi$ is a solution to
$- \omega Q  + \Delta Q + |Q|^{p-1}Q - |Q|^\frac4d Q = 0$.

 The minimality of $\cS_\omega(Q)$ among the solutions is clear from \eqref{eq1.7}, since every solution $Q$ in $H^1(\R^d)$ of \eqref{eq2.8} satisfies
$\K(Q) = \cS_\omega'(Q) \LL Q = 0$. \qquad $\Box$

\vskip 0.2cm

  We now turn to find the ground state in the case $p = 1 + \tfrac4{d-2}, \, d \ge 3$, that is to prove Proposition \ref{pr1.1}.

 Let
 \begin{align}
 \K^0(\varphi) & = \int_{\R^d}\Big( |\nabla \varphi|^2 - |\varphi|^\frac{2d}{d-2}\Big)\,\mathrm{d}x,\label{eq2.9v2}\\
   \cH^0(\varphi) &  = \tfrac1d \|\varphi\|_{L^\frac{2d}{d-2}}^\frac{2d}{d-2}.\label{eq2.10v2}
\end{align}
We will show
\begin{lem}\label{le2.6v2}
For $p = 1 + \frac4{d-2}, \, d \ge 3$,
\EQ{ \label{eq2.11v2}
m_\omega & = \inf\{ H^0(\varphi): \varphi\in H^1\setminus \{0\}, \K^0(\varphi) < 0\}\\
          & = \inf\{ H^0(\varphi): \varphi\in H^1\setminus \{0\}, \K^0(\varphi) \le 0\}.}
\end{lem}
\begin{proof}
Since $\K^0(\varphi) \le \K(\varphi)$, $\cH^0(\varphi) \le \cH_\omega(\varphi)$, it follows that
\begin{align*}
m_\omega = \inf\{ \cH_\omega(\varphi): \varphi\in H^1\setminus\{0\}, \K(\varphi) < 0\}
         \ge \inf\{ \cH^0(\varphi): \varphi\in H^1\setminus \{0\}, \K^0(\varphi) < 0\}.
\end{align*}
Hence, in order to show the first equality, it suffices to show
\begin{equation}\label{eq2.12v2}
\inf\{ \cH_\omega(\varphi): \varphi \in H^1 \setminus \{0\}, \K(\varphi) < 0\}
\le \inf\{ \cH^0(\varphi): \varphi \in H^1\setminus \{0\}, \K^0(\varphi) < 0\}.
\end{equation}
For any $\varphi \in H^1\setminus\{0\}$ with $\K^0(\varphi) < 0$, taking $\widetilde{T}_\lambda \varphi(x) = \lambda^\frac{d-2}2 \varphi(\lambda x)$,
we have $\text{as }  \lambda \to \infty$,
\begin{align*}
\K(\widetilde{T}_\lambda \varphi) & = \int_{\R^d}\Big( |\nabla \varphi|^2 - |\varphi|^\frac{2d}{d-2} + \tfrac{d}{d+2} \lambda^{-\frac4d}
|\varphi|^\frac{2(d+2)}d\Big) \,\mathrm{d}x \to \K^0(\varphi),\\
\cH_\omega(\widetilde{T}_\lambda \varphi) & = \tfrac{\omega}2
\lambda^{-2} \|\varphi\|_{L^2}^2  + \tfrac1d
\|\varphi\|_{L^\frac{2d}{d-2}}^\frac{2d}{d-2} \to \cH^0(\varphi).
\end{align*}
This gives \eqref{eq2.12v2} and completes the proof of the first equality.

  For the second equality, it suffices to show
\begin{equation} \label{eq2.13v2}
\inf\{\cH^0(\varphi): \varphi\in H^1\setminus \{0\}, \K^0(\varphi) < 0\}\le \inf\{ \cH^0(\varphi): \varphi \in H^1\setminus \{0\}, \K^0(\varphi) \le 0\}.
\end{equation}
For any $\varphi \in H^1\setminus\{0\}$ with $\K^0(\varphi) \le 0$ and
 \begin{align*}
\LL \,\K^0(\varphi)  = \int \Big(2|\nabla \varphi|^2 - \tfrac{2d}{d-2}
|\varphi|^\frac{2d}{d-2} \Big)\,\mathrm{d}x
                     = 2 \K^0(\varphi) - \tfrac4{d-2} \|\varphi\|_{L^\frac{2d}{d-2}}^\frac{2d}{d-2} < 0,
\end{align*}
which implies
$\K^0(T_\lambda \varphi)  < 0,   \text{ for } \lambda>1.$
We also have
\begin{align*}
\cH^0(T_\lambda \varphi)  = \tfrac1d \lambda^\frac{2d}{d-2}
\|\varphi\|_{L^\frac{2d}{d-2}}^\frac{2d}{d-2}
\to \cH^0(\varphi),  \text{ as }  \lambda \to 1,
\end{align*}
so we obtain \eqref{eq2.13v2} and complete the proof.
\end{proof}
{\it Proof of Proposition \ref{pr1.1}.}

By Lemma \ref{le2.6v2}, we have
\begin{align}
m_\omega & = \inf\left\{ \tfrac1d \|\varphi\|_{L^\frac{2d}{d-2}}^\frac{2d}{d-2}: \varphi \in H^1\setminus \{0\}, \|\nabla \varphi\|_{L^2}^2 \le \|\varphi\|_{L^\frac{2d}{d-2}}^\frac{2d}{d-2}\right\} \nonumber \\
 & \ge \inf\left\{ \tfrac1d\|\nabla \varphi\|_{L^2}^2: \varphi
 \in H^1\setminus\{0\}, \|\nabla \varphi\|_{L^2}^2 \le \|\varphi\|_{L^\frac{2d}{d-2}}^\frac{2d}{d-2}\right\} \nonumber \label{eq2.18v7}\\
& \ge  \inf\left\{\tfrac1d \|\nabla \varphi\|_{L^2}^2 \left(\tfrac{\|\nabla \varphi\|_{L^2}^2}{ \|\varphi\|_{L^\frac{2d}{d-2}}^\frac{2d}{d-2}}\right)^\frac{d-2}2: \varphi\in H^1\setminus \{0\}\right\} \nonumber \\
& \ge \inf \left\{ \tfrac1d \left(\tfrac{\|\nabla
\varphi\|_{L^2}}{\|\varphi \|_{L^\frac{2d}{d-2}}}\right)^d:
\varphi\in \dot{H}^1\setminus \{0\}\right\} = \tfrac1d(C_d^*)^{-d},
\nonumber
\end{align}
where $C_d^*$ is the sharp Sobolev constant in $\R^d$, that is
\begin{equation}\label{eqsobolev}
\|\varphi\|_{L^\frac{2d}{d-2}(\R^d)} \le C_d^* \|\nabla \varphi\|_{L^2}, \ \forall\, \varphi \in \dot{H}^1(\R^d),
\end{equation}
with the equality is attained by $W$ (see \cite{Aubin},
\cite{Talenti}) and $\tfrac1d(C_d^*)^{-d} = \E^0(W)$.

On the other hand, by the density $H^1(\R^d)$ in $\dot{H}^1(\R^d)$, we can find $\varphi_n\in H^1\setminus \{0\}$, such that $\varphi_n \to W \text{ in } \dot{H}^1(\R^d), \text{  as
} n \to \infty$. Then we have $\cH^0(\varphi_n) \to \cH^0(W) = \E^0(W), \text{ as } n\to \infty$,
by Lemma \ref{le2.6v2}, $\cH^0(\varphi_n) \ge m_\omega$ for $n$ large enough. So we have $m_\omega \le \E^0(W)$.
Thus, we obtain \eqref{eq2.14v2}.
 \hspace{5cm} $\Box$

Next we show the energy-trapping properties of $A_{\omega,\pm}$.
\begin{lem}\label{le2.6}
For $1 + \tfrac4d < p \le 1 + \tfrac4{d-2}, \, d\ge 3, \text{ and }
1 + \tfrac4d < p < \infty, \, d = 1, 2$, we have for any $ \fy \in
H^1(\R^d)$ with $\K(\fy) \ge 0$, \EQ{  \label{eq2.9}
  \tfrac{d(p-1)-4}{d(p-1)} \int \tfrac12|\na \fy|^2 + \tfrac{d}{2(d+2)} |\fy|^\frac{2(d+2)}d \,\mathrm{d}x
\le   \E(\fy)   \le \int \tfrac12|\na \fy|^2 + \tfrac{d}{2(d+2)}
|\fy|^\frac{2(d+2)}d \,\mathrm{d}x. }
\end{lem}
\begin{proof}
On the one hand,
\begin{align*}
\E(\varphi) & = \int \tfrac12 |\nabla \varphi|^2 - \tfrac1{p+1} |\varphi|^{p+1} + \tfrac{d}{2(d+2)} |\varphi|^\frac{2(d+2)}d\,\mathrm{d}x\\
            & \le \int \tfrac12 |\nabla \varphi|^2 + \tfrac{d}{2(d+2)} |\varphi|^\frac{2(d+2)}d \,\mathrm{d}x.
\end{align*}
On the other hand, since
\begin{equation*}
\K(\varphi) = \int |\nabla \varphi|^2 - \tfrac{d(p-1)}{2(p+1)}
|\varphi|^{p+1} + \tfrac{d}{d+2} |\varphi|^\frac{2(d+2)}d
\,\mathrm{d}x \ge 0,
\end{equation*}
we have
\begin{equation*}
\tfrac1{p+1} \int |\varphi|^{p+1} \,\mathrm{d}x \le \tfrac2{d(p-1)}
\int |\nabla \varphi|^2 + \tfrac{d}{d+2} |\varphi|^\frac{2(d+2)}d
\,\mathrm{d}x.
\end{equation*}
So
\begin{align*}
\E(\varphi) & = \int \tfrac12|\nabla \varphi|^2 - \tfrac1{p+1} |\varphi|^{p+1} + \tfrac{d}{2(d+2)} |\varphi|^\frac{2(d+2)}d \,\mathrm{d}x\\
             & \ge  \int \tfrac12 |\nabla \varphi|^2 + \tfrac{d}{2(d+2)} |\varphi|^\frac{2(d+2)}d \,\mathrm{d}x
              - \tfrac2{d(p-1)} \int  |\nabla \varphi|^2 + \tfrac{d}{d+2} |\varphi|^\frac{2(d+2)}d \,\mathrm{d}x\\\
             &  = \int \left(\tfrac12 - \tfrac2{d(p-1)}\right) |\nabla \varphi|^2 + \left(\tfrac{d}{2(d+2)} - \tfrac2{(d+2)(p-1)}\right) |\varphi|^\frac{2(d+2)}d \,\mathrm{d}x\\
              & = \tfrac{d(p-1)-4}{d(p-1)} \int \tfrac12|\na \fy|^2 + \tfrac{d}{2(d+2)} |\fy|^\frac{2(d+2)}d \,\mathrm{d}x.
\end{align*}
\end{proof}
\begin{prop}[Energy-trapping for $A_{\omega,-}$]\label{le2.7}
 For $1 + \tfrac4d < p \le  1 + \tfrac4{d-2}, \, d \ge 3, \  1 + \tfrac4d < p < \infty, \, d = 1, 2$,  and $u_0\in A_{\om,-}$.
Let $u$ be the solution of \eqref{eq1.1}, and $I_{max}$ be the lifespan of $u$, then
\begin{equation} \label{eq2.11}
\K(u(t)) < - \big(m_\om - \cS_\om(u(t))\big), \ \ \forall \,t\in I_{max}.
\end{equation}
\end{prop}
\begin{proof}
We first claim that $\K(u(t)) < 0, \text{ for } t\in I_{max}$.
Indeed, since $u_0\in A_{\om,-}$,
we have by the mass and energy conservation
 that
$\cS_\omega(u(t)) < m_\omega, \text{ for }  t \in I_{max}$.
If $\K(u(t_0)) \ge 0$ for some $t_0 \in I_{max}$, then there is $t_1 \in I_{max}$, such that $\K(u(t_1)) = 0$. So we have $\cS_\omega(u(t_1)) \ge m_\omega$, which contradicts $\cS_\omega(u(t)) < m_\omega, \text{ for all }  t\in I_{max}$, so we have $\K(u(t)) < 0 \text{ for } t\in I_{max}$.

Next, we turn to \eqref{eq2.11}. By the above claim, for any $ t \in I_{max}$, there exists $0 < \lambda(t) < 1$ such that $\K(T_{\lambda(t)} u(t)) = 0$,
which together with the definition of $m_\omega$ shows that $\cS_\omega(T_{\lambda(t)} u(t)) \ge m_\omega$.
By Lemma \ref{le2.3},
\begin{align*}
\LL^2 \,\cS_\omega(u(t))
& = 2 \LL \,\cS_\omega(u(t)) - \tfrac{d(p-1)(d(p-1)-4)}{4(p+1)} \|u(t)\|_{L^{p+1}}^{p+1}\\
& = 2 \K(u(t)) - \tfrac{d(p-1)(d(p-1)-4)}{4(p+1)}
\|u(t)\|_{L^{p+1}}^{p+1} < 0,
\end{align*}
we have
\begin{align*}
\cS_\omega(u(t))&  > \cS_\omega(T_{\lambda(t)} u(t)) + (1
-\lambda(t))\tfrac{\mathrm{d}}{\mathrm{d}\lambda}\big|_{\lambda = 1}
\cS_\omega(T_\lambda u(t))\\
                  &  = \cS_\omega(T_{\lambda(t)} u(t)) + (1 - \lambda(t)) \K(u(t)) \\
                  &  > m_\omega + \K(u(t)),
\end{align*}
so
$\K(u(t)) < - (m_\omega - \cS_\omega(u(t)))$.
\end{proof}
Before discussing the energy-trapping for $A_{\om,+}$, we first show $\forall\, \om > 0$, $A_{\om,+}$ is bounded in $H^1(\R^d)$.
\begin{lem}\label{le2.8}
 Let $\om > 0$ and $u \in A_{\om,+}$, then we have
\begin{equation*}
\norm{u}_{L^2}^2 \le \tfrac{2m_{\om}}\om, \  \norm{\na u}_{L^2}^2
\les m_\om,
\end{equation*}
and hence
\begin{equation} \label{eq2.12}
\norm{u}_{H^1}^2 \les m_\om + \tfrac{m_\om}\om.
\end{equation}
\end{lem}
\begin{proof}
By $u\in A_{\om,+}$, we have
\begin{align*}
& \cH_\omega(u) \le \cS_\omega(u) \le m_\omega,
\end{align*}
which implies
\begin{equation*}
\|u\|_{L^2}^2 \le \tfrac{2 m_\omega}\omega.
\end{equation*}
The boundedness of $\|\nabla u\|_{L^2}$ follows from
\begin{align*}
m_\omega\ge \cS_\omega(u) & \ge \E(u) = \int \tfrac12|\nabla u|^2 - \tfrac1{p+1} |u|^{p+1} + \tfrac{d}{2(d+2)} |u|^\frac{2(d+2)}d\,\mathrm{d}x\\
                &  \ge \int \tfrac12 |\nabla u|^2 + \tfrac{d}{2(d+2)} |u|^\frac{2(d+2)}d \,\mathrm{d}x
               - \tfrac2{d(p-1)} \left( \|\nabla u\|_{L^2}^2 + \tfrac{d}{d+2} \|u\|_{L^\frac{2(d+2)}d}^\frac{2(d+2)}d\right)\\
                                & \ge \left(\tfrac12 - \tfrac2{d(p-1)} \right) \|\nabla u\|_{L^2}^2,
\end{align*}
where the first inequality is given by $\K(u) \ge 0$.
\end{proof}

\begin{prop}[Energy-trapping for $A_{\omega,+}$]\label{le2.9}
For $u_0\in A_{\omega,+}$, let $u$ be a solution of \eqref{eq1.1}, with $I_{max}$ the lifespan, we have
$\delta > 0$ depending on $d$, $p$ and $\omega$ such that for $t \in I_{max}$,
\begin{equation} \label{eq2.14}
\K(u(t)) \ge \min\left\{\tfrac{d(p-1)-4}{d(p-1)}\Big( \norm{\na
u(t)}_{L^2}^2 + \tfrac{d}{d+2}
\norm{u(t)}_{L^\frac{2(d+2)}d}^\frac{2(d+2)}d\Big), \delta
\Big(m_\om - \cS_\om(u(t))\Big)\right\}.
\end{equation}
\end{prop}
\begin{proof}Direct computation shows that
\begin{align*}
\cS_\omega(T_\lambda u(t))&  = \tfrac{\omega}2 \|T_\lambda u(t) \|_{L^2}^2 + \E(T_\lambda u(t)) \\
                         & = \tfrac{\omega}2 \|u(t)\|_{L^2}^2 + \int \tfrac12 \lambda^2 |\nabla u(t)|^2 + \tfrac{d}{2(d+2)}
                          \lambda^2 |u(t)|^\frac{2(d+2)}d - \tfrac1{p+1} \lambda^{\frac{d}2(p-1)} |u(t)|^{p+1} \,\mathrm{d}x
\end{align*}
and
\begin{align*}
\tfrac{\mathrm{d}^2}{\mathrm{d} \lambda^2} \cS_\omega (T_\lambda
u(t))
& = \int |\nabla u(t)|^2 + \tfrac{d}{d+2} |u(t)|^\frac{2(d+2)}d -
\tfrac{d(p-1)(d(p-1)-2)}{4(p+1)} \lambda^{\frac{d}2(p-1)-2}
|u(t)|^{p+1} \,\mathrm{d}x.
\end{align*}
Since
\begin{align*}
\K(T_\lambda u(t))
& = \int \lambda^2 |\nabla u(t)|^2 - \tfrac{d(p-1)}{2(p+1)}
\lambda^{\frac{d}2(p-1)} |u(t)|^{p+1} + \tfrac{d}{d+2} \lambda^2
|u(t)|^\frac{2(d+2)}d \,\mathrm{d}x,
\end{align*}
we have
\begin{align}
& \tfrac{\mathrm{d}^2}{\mathrm{d}\lambda^2} \cS_\omega(T_\lambda u(t))
= - \tfrac1{\lambda^2} \K(T_\lambda u(t)) + \tfrac2{\lambda^2}
\left(\K(T_\lambda u(t)) + \tfrac{(4-d(p-1))d(p-1)}{8(p+1)} \int
|T_\lambda u(t)|^{p+1} \,\mathrm{d}x\right). \label{eq2.21v2}
\end{align}
{\bf Case I.}\begin{equation}\label{eq2.22v2} \K(u(t)) -
\tfrac{(d(p-1)-4)d(p-1)}{8(p+1)} \int |u(t)|^{p+1}
\,\mathrm{d}x \ge 0, \  \text{for}\ t \in I_{max}.
\end{equation}
In this case,
\begin{align*}
\K(u(t)) &  = \int |\nabla u|^2 - \tfrac{d(p-1)}{2(p+1)} |u|^{p+1} + \tfrac{d}{d+2} |u|^\frac{2(d+2)}d \,\mathrm{d}x\\
         & \ge \int |\nabla u|^2 + \tfrac{d}{d+2} |u|^\frac{2(d+2)}d \,\mathrm{d}x - \tfrac4{d(p-1)-4} \K(u(t)),
\end{align*}
thus
\begin{equation*}
\begin{split}
\K(u(t)) & \ge \tfrac{d(p-1)-4}{d(p-1)} \int |\nabla u(t)|^2 + \tfrac{d}{d+2} |u(t)|^\frac{2(d+2)}d \,\mathrm{d}x, \ \forall\,  t\in I_{max}.
\end{split}
\end{equation*}
{\bf Case II.}
\begin{equation} \label{eq2.25v2}
\mathcal{K}(u(t)) - \tfrac{(d(p-1)-4)d(p-1)}{8(p+1)} \int |u(t)|^{p+1}
\,\mathrm{d}x < 0,\ \ \text{for}\ t\in I_{max}.
\end{equation}
In this case,
\begin{equation*}
\|\nabla u(t)\|_{L^2}^2 < \tfrac{d^2 (p-1)^2}{8(p+1)} \int
|u(t)|^{p+1} \,\mathrm{d}x - \tfrac{d}{d+2} \int
|u(t)|^\frac{2(d+2)}d \,\mathrm{d}x,
\end{equation*}
which implies $u(t)\ne 0, \,\forall \,t \in I_{max}$, and
\begin{equation}\label{eq2.26v2}
\tfrac{d^2 (p-1)^2}{8(p+1)} \|u(t)\|_{L^{p+1}}^{p+1} \ge \|\nabla
u(t)\|_{L^2}^2.
\end{equation}
 Since $u_0 \in A_{\omega, +}$ and $u(t)\ne 0, \,\forall \,t \in I_{max}$, together with the definition of $m_\omega$, we have $\mathcal{K}(u(t)) > 0$ by similar argument as in the claim in Proposition \ref{le2.7}. Therefore, by Lemma \ref{le2.1}, there is
 $\lambda(t) > 1$ such that
 \begin{equation} \label{eq2.28v2}
 \K(T_{\lambda(t)} u(t)) = 0
\end{equation}
and
\begin{equation}\label{eq2.30v7}
\K(T_\lambda u(t)) > 0 \text{ for } 1 \le \lambda < \lambda(t).
\end{equation}
By \eqref{eq2.28v2}, we have
\begin{equation*}
\|\nabla u(t)\|_{L^2}^2 - \tfrac{d(p-1)}{2(p+1)}
\lambda(t)^{\frac{d}2(p-1)-2} \|u(t)\|_{L^{p+1}}^{p+1} +
\tfrac{d}{d+2} \|u(t)\|_{L^\frac{2(d+2)}d}^\frac{2(d+2)}d = 0,
\end{equation*}
so by Lemma \ref{le2.8} together with the interpolation and Sobolev inequalities, we have
\begin{align*}
\lambda(t)^{\frac{d}2(p-1)-2} \|u(t)\|_{L^{p+1}}^{p+1}
& =  \tfrac{2(p+1)}{d(p-1)}\left(\|\nabla u(t)\|_{L^2}^2 + \tfrac{d}{d+2}\|u(t)\|_{L^\frac{2(d+2)}d}^\frac{2(d+2)}d\right) \\
& \le \tfrac{2(p+1)}{d(p-1)}\left(\|\nabla u(t)\|_{L^2}^2 + C_d \|u(t)\|_{L^2}^\frac4d \|\nabla u(t)\|_{L^2}^2\right)\\
& \le \tfrac{2(p+1)}{d(p-1)}\left(1+ C_d
\Big(\tfrac{2m_\omega}\omega\Big)^\frac2d  \right)  \|\nabla
u(t)\|_{L^2}^2,
\end{align*}
where $C_d$ is some constant depending on $d$ related to the Sobolev inequality.
Thus, we see by \eqref{eq2.26v2},
\begin{align*}
\tfrac{2(p+1)}{d(p-1)}\left(1+ C_d \Big(\tfrac{2m_\omega}\omega\Big)^\frac2d  \right) \lambda(t)^2\|\nabla u(t)\|_{L^2}^2 &
 \ge \lambda(t)^{\frac{d}2(p-1)} \|u(t)\|_{L^{p+1}}^{p+1}\\
                                       &  \ge \tfrac{8(p+1)}{d^2(p-1)^2} \lambda(t)^{\frac{d}2(p-1)}\|\nabla u(t)\|_{L^2}^2,
\end{align*}
which yields
\begin{align}\label{eq2.29v2}
\lambda(t) \le \left(\tfrac{d}4(p-1)\bigg(1 + C_d
\Big(\tfrac{2m_\omega}\omega\Big)^\frac2d
\bigg)\right)^\frac1{\frac{d}2(p-1)-2}.
\end{align}
Because
\begin{align*}
& \tfrac{\mathrm{d}}{\mathrm{d}\lambda}\left(\tfrac1{\lambda^2} \K(T_\lambda u)\right) = - \tfrac{d(p-1)}{2(p+1)} \left( \tfrac{d}2(p-1) -2\right) \lambda^{\frac{d}2(p-1)-3} \int |u|^{p+1} \,\mathrm{d}x \le 0,\\
& \tfrac{\mathrm{d}}{\mathrm{d}\lambda}\left( \tfrac1{\lambda^2}
\int |T_\lambda u|^{p+1} \,\mathrm{d}x\right)   =
\left(\tfrac{d}2(p-1) -2\right) \int |u|^{p+1} \,\mathrm{d}x \ge 0,
\end{align*}
we have
\begin{equation} \label{eq2.30v2}
\tfrac{\mathrm{d}}{\mathrm{d}\lambda} \left( \tfrac1{\lambda^2}
\Big( \K(T_\lambda u) + \tfrac{(4-d(p-1))d(p-1)}{8(p+1)} \int
|T_\lambda u|^{p+1} \,\mathrm{d}x\Big)\right) \le 0.
\end{equation}
Collecting \eqref{eq2.25v2} and \eqref{eq2.30v2}, we have
\begin{align}
 \tfrac1{\lambda^2}(\K(T_\lambda u(t)) + \tfrac{(4-d(p-1))d(p-1)}{8(p+1)} \int |T_\lambda u(t)|^{p+1}
  \,\mathrm{d}x) < 0, \text{  for } \lambda \ge 1. \label{eq2.31v2}
\end{align}
Hence, \eqref{eq2.21v2}, \eqref{eq2.30v7} and \eqref{eq2.31v2} shows for $1 \le \lambda \le \lambda(t)$,
\begin{align}
  & \tfrac{\mathrm{d}^2}{\mathrm{d}\lambda^2} \cS_\omega(T_\lambda u(t))\nonumber\\
 = &- \tfrac1{\lambda^2} \K(T_\lambda u(t)) + \tfrac2{\lambda^2}\left( \K(T_\lambda u(t)) + \tfrac{(4-d(p-1))d(p-1)}{8(p+1)} \int |T_\lambda u(t)|^{p+1} \,\mathrm{d}x\right)\nonumber \\
 <  & - \tfrac1{\lambda^2} \K(T_\lambda u(t)) \le 0.\label{eq2.34v2}
\end{align}
Combining \eqref{eq2.29v2} and \eqref{eq2.34v2}, we obtain
\begin{align*}
& \left(\left(\tfrac{d}4(p-1)\bigg(1 +  C_d \Big(\tfrac{2m_\omega}\omega\Big)^\frac2d\bigg)\right)^\frac1{\frac{d}2(p-1)-2}-1\right) \K(u(t)) \nonumber\\
          & \ge  (\lambda(t)-1)\tfrac{\mathrm{d}}{\mathrm{d}\lambda}\big|_{\lambda =1} \cS_\omega(T_\lambda u(t))\nonumber \\
          & \ge \cS_\omega(T_{\lambda(t)} u(t)) - \cS_\omega(u(t))\nonumber\\
         &  \ge m_\omega - \cS_\omega(u(t)). \nonumber
\end{align*}
Thus, there is $\delta > 0$ depending on $d, \, p$ and $\omega$ that
\begin{equation*}
\K(u(t)) \ge \min\left\{\tfrac{d(p-1)-4}{d(p-1)}\Big( \norm{\na
u(t)}_{L^2}^2 + \tfrac{d}{d+2}
\norm{u(t)}_{L^\frac{2(d+2)}d}^\frac{2(d+2)}d\Big), \delta
\Big(m_\om - \cS_\om(u(t))\Big)\right\}.
\end{equation*}
\end{proof}

\section{Wellposedness and perturbation theory}
  In this section, we present the local wellposedness theory and the perturbation theory for \eqref{eq1.1}. We start by recording the
wellposedness theory. For the proof we refer to \cite{Cazenave, Cazenave-Weissler, Killip-Visan1}.
\begin{prop}
\label{th4.1}
\begin{enumerate}
\item[{\it (i)}] (Local existence) Let $\phi \in H^{1}(\mathbb{R}^d)$, $I$ be an interval, $t_0\in I$ and $A>0$. Assume that
\begin{equation*}
\left\|\phi \right\|_{H^1}\le A,
\end{equation*}
and there exists $\delta>0$ depending on $A$ that
\begin{equation*}
\left\|\jb{\nabla} e^{i(t-t_0)\Delta}\phi
\right\|_{L_{t,x}^\frac{2(d+2)}d(I\times \R^d)} \le \delta,
\end{equation*}
then there exists a unique solution $u \in C(I,H^1(\mathbb{R}^{d}))$ to \eqref{eq1.1} such that
\begin{align*}
u(t_0)=& \phi,\\
\left\| u \right\|_{S^1(I)} \lesssim & \left\| \phi \right\|_{H^1},\\
\left\|  \jb{\nabla} u \right\|_{L_{t,x}^\frac{2(d+2)}d(I\times \R^d)} \le& 2 \left\| \jb{\nabla} e^{i(t-t_0)\Delta} \phi\right\|_{L_{t,x}^\frac{2(d+2)}d(I\times \R^d)}.
\end{align*}
As a consequence, we have the small data global existence:
if $\norm{\phi}_{H^1}$ is sufficiently small, then $u$ is a global solution with $\norm{u}_{S^1(\R)} \lesssim \norm{\phi}_{H^1}$.

\item[{\it (ii)}] (Unconditional uniqueness) Suppose $u_1, u_2 \in C(I,H^1(\mathbb{R}^d))$ are two solutions of
\eqref{eq1.1} with $u_{1}(t_0)= u_2(t_0)$ for some $t_{0} \in I$, then $u_1= u_2$.

Let $u \in C((T_{min}, T_{max}), H^1(\mathbb{R}^d))$ be the maximal-lifespan solution to \eqref{eq1.1}, then we have
\item[{\it (iii)}] (Conservation laws)
For any $t, t_0 \in I_{\max}$,
\begin{align*}
\mathcal{M}(u(t)) & = \mathcal{M}(u(t_0)),\\
\mathcal{E}(u(t))& =\mathcal{E}(u(t_0)),\\
\mathcal{S}_{\omega}(u(t))
&=\mathcal{S}_{\omega}(u(t_0)),\quad\mbox{for any $\omega > 0$},\\
\mathcal{P}(u(t))& = \mathrm{\Im} \int_{\mathbb{R}^{d}}\nabla u(t,x )\overline{u(t,x)}\,dx =\mathcal{P}(u(t_0)).
\end{align*}

\item[{\it (iv)}] (Blow-up criterion)
If $T_{\max}  < \infty$, then
\begin{equation*}
\left\|u \right\|_{
L_{t,x}^\frac{2(d+2)}d \cap  L_{t,x}^\frac{(d+2)(p-1)}2([T,\, T_{\max})\times \R^d)}= \infty,
\ \forall\, T_{min} < T < T_{max}.
\end{equation*}
A similar result holds, if $T_{\min}   > -\infty$.
\item[{\it (v)}](Scattering) If
\begin{equation}
\label{eq4.11}
\left\|u \right\|_{L_{t,x}^\frac{2(d+2)}d \cap  L_{t,x}^\frac{(d+2)(p-1)}2((T_{min}, T_{max}) \times \R^d)} <\infty,
\end{equation}
then $T_{\max}= \infty$, $T_{\min}= - \infty$,
 and there exist $u_{\pm}\in H^{1}(\mathbb{R}^d)$ such that
\begin{equation}\label{eq4.12}
\lim_{t\to \infty}
\left\|u(t)-e^{it\Delta}u_+ \right\|_{H^1}=\lim_{t\to -\infty}\left\|u(t) -e^{it\Delta}u_- \right\|_{H^1}=0.
\end{equation}
\end{enumerate}
\end{prop}
In the following, we will give the long-time perturbation theory when $d\le 4$.
\begin{prop}[Long-time perturbation]\label{pr4.3}
  Let $I$ be a compact time interval and let $w$ be an approximate solution to \eqref{eq1.1} on $I\times \R^d$ in the sense that
\begin{equation*}
i\partial_t w + \Delta w = |w|^\frac4d w - |w|^{p-1} w + e
\end{equation*}
 for some function $e$.

   Assume that
\begin{align}
\norm{w}_{L_t^\infty H_x^1(I\times \R^d)} \le A_1,\\\label{chuzjias}
\norm{w}_{L_{t,x}^\frac{2(d+2)}d \cap
L_{t,x}^\frac{(d+2)(p-1)}2(I\times \R^d)} \le B
\end{align}
for some $A_1, B > 0$.

   Let $t_0\in I$ and $u(t_0)$ close to $w(t_0)$ in the sense that
\begin{equation}\label{eq1.34v11}
\normo{u(t_0) - w(t_0)}_{H_x^1(\R^d)} \le A_2
\end{equation}
for some $A_2 > 0$.

  Assume also the smallness conditions
\begin{align}
\normo{\jb{\nabla}e^{i(t-t_0)\Delta}(u(t_0)- w(t_0))}_{L_{t,x}^\frac{2(d+2)}d(I\times \R^d)} \le \delta, \label{eq1.35v11} \\
\normo{\jb{\nabla} e}_{L_{t,x}^\frac{2(d+2)}{d+4}(I\times \R^d)} \le \delta
\end{align}
for some $0 < \delta \le \delta_1$, where $\delta_1 = \delta_1(A_1, A_2, B)$ is a small constant. Then there exists a solution $u$ to \eqref{eq1.1} on
$I\times \R^d$ with the specified initial data $u(t_0)$ at time $t = t_0$ that satisfies
\begin{align}
\normo{u - w}_{L_{t,x}^\frac{2(d+2)}d \cap L_{t,x}^\frac{(d+2)(p-1)}2 ( I \times \R^d)} \le& C(A_1, A_2, B) \delta^\alpha,\\
\normo{u - w}_{S^1(I)} \le& C(A_1, A_2, B) A_2,\\
\normo{ u}_{S^1(I)} \le& C(A_1, A_2, B),
\end{align}
where $0 < \alpha < \tfrac4{d(p-1)}$.
  \end{prop}
To show the long-time perturbation theory, we will first give the following short-time perturbation theory.
\begin{lem}[Short-time perturbation]\label{le3.3v11}
  Let $I$ be a compact time interval and let $w$ be an approximate solution to \eqref{eq1.1} on $I \times \R^d$ in the sense that
\begin{equation*}
i\partial_t w + \Delta w = |w|^\frac4d w - |w|^{p-1} w + e
\end{equation*}
for some function $e$.

   Suppose we also have the energy bound
\begin{equation}\label{equ3.11}
\normo{w}_{L_t^\infty H_x^1(I \times \R^d)} \le A_1
\end{equation}
for some constant $A_1 > 0$.

   Let $t_0 \in I$ and let $u(t_0) \in H^1(\R^d)$ be close to $w(t_0)$ in the sense that
\begin{equation}
\normo{u(t_0)- w(t_0)}_{H_x^1} \le A_2
\end{equation}
for some $A_2 > 0$.

  Moreover, assume the smallness conditions
\begin{align}\label{chuzsm}
\normo{\jb{\nabla} w}_{L_{t,x}^\frac{2(d+2)}d(I\times \R^d)} + \normo{w}_{   L_{t,x}^\frac{(d+2)(p-1)}2(I \times \R^d)} \le \delta_0,\\
\normo{\jb{\nabla} e^{i(t-t_0)\Delta} ( u(t_0) - w(t_0))}_{L_{t,x}^\frac{2(d+2)}d(I \times \R^d)} \le \delta,\\
\normo{\jb{\nabla} e}_{L_{t,x}^\frac{2(d+2)}{d+4}(I\times \R^d)} \le \delta
\end{align}
for some $0 < \delta \le \delta_0$, where $\delta_0  = \delta_0(A_1, A_2) > 0$ is a small constant.

  Then, there exists a solution $u \in S^1(I)$ to \eqref{eq1.1} on $I\times \R^d$ with the specified initial data $u(t_0)$ at time $t = t_0$ that
satisfies
\begin{align}
\normo{u-w}_{L_{t,x}^\frac{2(d+2)}d \cap L_{t,x}^\frac{(d+2)(p-1)}2(I \times \R^d)} \lesssim \delta^\alpha, \label{eq1.28v11}\\
\normo{u-w}_{S^1(I)} \lesssim A_2 + \delta^\alpha, \label{eq1.29v11}\\
\normo{u}_{S^1(I)} \lesssim A_1 + A_2, \label{eq1.30v11}\\
\normo{\jb{\nabla}((i\partial_t + \Delta)(u-w) + e)}_{L_{t,x}^\frac{2(d+2)}{d+4}(I\times \R^d)} \lesssim \delta^\alpha, \label{eq1.31v11}
\end{align}
where $0 < \alpha < \frac4{d(p-1)}$.
\end{lem}
\begin{proof}
By the wellposedness theory, it suffices to prove \eqref{eq1.28v11}-\eqref{eq1.31v11} as a priori estimate, that is, we assume that the solution $u$ already exists and belongs to $S^1(I)$.
  By time symmetry, we may assume $t_0 = \inf I$.

  Let $v = u - w$, then $v$ satisfies
\begin{equation*}
i\partial_t v + \Delta v = | w + v|^\frac4d ( w + v) - |w + v|^{p-1}(w + v) - |w|^\frac4d w + |w|^{p-1} w - e,
\end{equation*}
 and $v(t_0) = u(t_0) - w(t_0)$.

  For $T \in I$, define
\begin{equation*}
S(T) = \normo{\jb{\nabla}((i\partial_t + \Delta) v + e)}_{L_{t,x}^\frac{2(d+2)}{d+4}([t_0, T]\times \R^d)}.
\end{equation*}
 We will now work entirely on the slab $[t_0, T] \times \R^d$.
\begin{align}\nonumber
&\big\|\langle\nabla\rangle
v\big\|_{L_{t,x}^\frac{2(d+2)}d}+\norm{v}_{
L_{t,x}^\frac{(d+2)(p-1)}2}
\\\nonumber  \le& \big\|\langle\nabla\rangle e^{i(t-t_0)\Delta} v(t_0)\big\|_{L_{t,x}^\frac{2(d+2)}d}+ \normo{e^{i(t-t_0)\Delta} v(t_0)}_{L_{t,x}^\frac{(d+2)(p-1)}2}  + \normo{\jb{\nabla} e}_{L_{t,x}^\frac{2(d+2)}{d+4}}
\\\nonumber &+ \normo{\jb{\nabla}((i\partial_t + \Delta) v +
e)}_{L_{t,x}^\frac{2(d+2)}{d+4}}\\\label{vtgj}
 \lesssim& S(T) + \delta^\alpha, \ (0 < \alpha < \tfrac4{d(p-1)}< 1)
\end{align}
where we use the fact
\begin{align*}
 & \normo{e^{i(t-t_0)\Delta} v(t_0)}_{L_{t,x}^\frac{(d+2)(p-1)}2}\\
\lesssim & \normo{\jb{\nabla} e^{i(t-t_0)\Delta}v(t_0)}_{L_t^\frac{(d+2)(p-1)}2 L_x^\frac{2d(d+2)(p-1)}{d(d+2)(p-1)- 8}}\\
\lesssim  & \normo{\jb{\nabla} e^{i(t-t_0)\Delta}
v(t_0)}_{L_{t,x}^\frac{2(d+2)}d}^\frac4{d(p-1)} \normo{\jb{\nabla}
e^{i(t-t_0) \Delta} v(t_0)}_{L_t^\infty L_x^2}^{1 - \frac4{d(p-1)}}
\lesssim \delta^\frac4{d(p-1)} A_2^{1 - \frac4{d(p-1)}} \le
\delta^\alpha.
\end{align*}
  On the other hand, since
\begin{equation*}
\begin{split}
|\nabla((i\partial_t + \Delta) v + e)| \lesssim& |\nabla w|
|v|^\frac4d + |\nabla v||w + v|^\frac4d + |\nabla w||v||w|^{\frac4d
-1} + |\nabla w||v|^{p-1}\\ &+ |\nabla v||w + v|^{p-1} + |\nabla
w||v||w|^{p-2}
\end{split}
\end{equation*}
and
\begin{equation*}
 |(i\partial_t + \Delta) v + e| \lesssim |w||v|^\frac4d + |w+v|^\frac4d |v| + |w||v|^{p-1} + |w+v|^{p-1} |v|,
\end{equation*}
we have
\begin{align*}
S(T)   \lesssim& \normo{\jb{\nabla} w}_{L_{t,x}^\frac{2(d+2)}d}
\normo{(v,w)}_{L_{t,x}^\frac{2(d+2)}d}^\frac4d+ \normo{\jb{\nabla}
v}_{L_{t,x}^\frac{2(d+2)}d}
\normo{(v,w)}_{L_{t,x}^\frac{2(d+2)}d}^\frac4d\\
& + \normo{\jb{\nabla} w}_{L_{t,x}^\frac{2(d+2)}d}
\normo{(v,w)}_{L_{t,x}^\frac{(d+2)(p-1)}2}^{p-1} +
\normo{\jb{\nabla} v}_{L_{t,x}^\frac{2(d+2)}d}
\normo{(v,w)}_{L_{t,x}^\frac{(d+2)(p-1)}2}^{p-1}
\\
 \lesssim& \delta_0 \big[(S(T) + \delta^\alpha)^\frac4d
+\delta_0^\frac4d + (S(T) + \delta^\alpha)^{p-1}+\delta_0^{p-1}\big]\\
& + (\delta_0^\frac4d + \delta_0^{p-1})(S(T) + \delta^\alpha) +
(S(T) + \delta^\alpha)^{1+ \frac4d} + (S(T) + \delta^\alpha)^p.
\end{align*}
By the continuity argument, we can take $\delta_0 = \delta_0(A_1,A_2)$ sufficiently small, then
\begin{equation}\label{stgj}
S(T) \lesssim \delta^\alpha, \ \forall\, T\in I,
\end{equation}
which implies \eqref{eq1.31v11}.
  We also have
\begin{align*}
& \normo{u-w}_{L_{t,x}^\frac{2(d+2)}d \cap L_{t,x}^\frac{(d+2)(p-1)}2} \\
\lesssim &  \normo{e^{i(t-t_0)\Delta} v(t_0)}_{L_{t,x}^\frac{2(d+2)}d \cap L_{t,x}^\frac{(d+2)(p-1)}2}
+ \normo{\jb{\nabla}((i\partial_t + \Delta) v +e)}_{L_{t,x}^\frac{2(d+2)}{d+4}} + \normo{\jb{\nabla} e}_{L_{t,x}^\frac{2(d+2)}{d+4}}
\lesssim \delta^\alpha,
\end{align*}
which is \eqref{eq1.28v11}.
  To obtain \eqref{eq1.29v11}, we see
\begin{align*}
\normo{u-w}_{S^1(I)} & \lesssim \normo{u(t_0) - w(t_0)}_{H^1} + \normo{\jb{\nabla}((i\partial_t + \Delta) v + e)}_{L_{t,x}^\frac{2(d+2)}{d+4}}
+ \normo{\jb{\nabla} e}_{L_{t,x}^\frac{2(d+2)}{d+4}}\\
& \lesssim A_2 + S(t) + \delta
\lesssim A_2 + \delta^\alpha.
\end{align*}

We now show \eqref{eq1.30v11}. Using Strichartz estimate,
\eqref{equ3.11} and \eqref{chuzsm}, we get
\begin{align*}
\normo{w}_{S^1(I)} & \lesssim \normo{w(t_0)}_{H^1} +
 \normo{\jb{\nabla}(|w|^\frac4d w - |w|^{p-1}w)}_{L_{t,x}^\frac{2(d+2)}{d+4}} + \normo{\jb{\nabla} e}_{L_{t,x}^\frac{2(d+2)}{d+4}(I \times \R^d)}\\
 & \lesssim A_1 + \normo{w}_{L_{t,x}^\frac{2(d+2)}d}^\frac4d \normo{\jb{\nabla} w}_{L_{t,x}^\frac{2(d+2)}d}
+ \normo{w}_{L_{t,x}^\frac{(d+2)(p-1)}2}^{p-1} \normo{\jb{\nabla} w}_{L_{t,x}^\frac{2(d+2)}d} + \delta\\
& \lesssim A_1 + (\delta_0^\frac4d +
\delta_0^{p-1})\normo{w}_{S^1(I)} + \delta.
\end{align*}
By the continuity argument, we have
\begin{equation*}
\normo{w}_{S^1(I)} \lesssim A_1,
\end{equation*}
provided $\delta_0$ is sufficiently small depending on $A_1$. This
together with \eqref{vtgj}, \eqref{stgj} and $u=v+w$ yields
\begin{align*}
 \normo{\jb{\nabla} u}_{L_{t,x}^\frac{2(d+2)}d}
  \le  \normo{\jb{\nabla}
v}_{L_{t,x}^\frac{2(d+2)}d} + \normo{\jb{\nabla}
w}_{L_{t,x}^\frac{2(d+2)}d} \lesssim A_1.
\end{align*}
While, we have by \eqref{eq1.28v11}, \eqref{vtgj} and \eqref{stgj}
\begin{align*} \normo{u}_{L_{t,x}^\frac{2(d+2)}d\cap
L_{t,x}^\frac{(d+2)(p-1)}2}\leq\normo{v}_{L_{t,x}^\frac{2(d+2)}d\cap
L_{t,x}^\frac{(d+2)(p-1)}2} + \normo{w}_{L_{t,x}^\frac{2(d+2)}d\cap
L_{t,x}^\frac{(d+2)(p-1)}2}\lesssim\delta^\al
\end{align*} Combining these with Strichartz estimate, we
obtain
\begin{align*}
\normo{u}_{S^1(I)} & \lesssim \normo{u(t_0)}_{H^1} + \normo{\jb{\nabla}(|u|^\frac4du - |u|^{p-1} u)}_{L_{t,x}^\frac{2(d+2)}{d+4}} \\
    & \lesssim A_1 + A_2 + \normo{u}_{L_{t,x}^\frac{2(d+2)}d}^\frac4d \normo{\jb{\nabla} u}_{L_{t,x}^\frac{2(d+2)}d}
   + \normo{u}_{L_{t,x}^\frac{(d+2)(p-1)}2}^{p-1} \normo{\jb{\nabla}u}_{L_{t,x}^\frac{2(d+2)}d}\\
    & \lesssim A_1 + A_2 + \delta^{ \frac4d\alpha}A_1 + \delta^{(p-1)
    \alpha}A_1
   \lesssim A_1 + A_2,
\end{align*}
which proves \eqref{eq1.30v11}, provided $\delta_0$ is sufficiently small depending on $A_1$ and $A_2$.
\end{proof}
We now show the long-time perturbation theory.

{\it Proof of Proposition \ref{pr4.3}}: We will derive Proposition
\ref{pr4.3} from Lemma \ref{le3.3v11} by an iterative procedure.
First, we will assume without loss of generality that $t_0 = \inf
I$. Let $\delta_0 = \delta_0(A_1, 2A_2)$ be as in Lemma
\ref{le3.3v11}.

The first step is to establish an $S^1$ bound on $w$. In order to do so, we subdivide $I$ into $N_0 \sim (1 + \frac{M}{\epsilon_0})^\frac{d}{2(d+2)}$ subintervals
$I_k$ such that
\begin{equation*}
\normo{w}_{L_{t,x}^\frac{2(d+2)}d \cap L_{t,x}^\frac{(d+2)(p-1)}2(I_k \times \R^d)} \sim \delta_0.
\end{equation*}
On each subinterval $I_k$, we have
\begin{align*}
\normo{w}_{S^1(I_k)} & \lesssim \normo{w}_{L_t^\infty H_x^1} + \normo{\jb{\nabla}(|w|^\frac4d w - |w|^{p-1}w)}_{L_{t,x}^\frac{2(d+2)}{d+4}} + \normo{\jb{\nabla} e}_{L_{t,x}^\frac{2(d+2)}{d+4}(I_k \times \R^d)}\\
 & \lesssim A_1 + \normo{w}_{L_{t,x}^\frac{2(d+2)}d}^\frac4d \normo{\jb{\nabla} w}_{L_{t,x}^\frac{2(d+2)}d}
+ \normo{w}_{L_{t,x}^\frac{(d+2)(p-1)}2}^{p-1} \normo{\jb{\nabla} w}_{L_{t,x}^\frac{2(d+2)}d} + \delta\\
& \lesssim A_1 + (\delta_0^\frac4d + \delta_0^{p-1})\normo{w}_{S^1(I_k)} + \delta.
\end{align*}
By the continuity argument, we have
$\normo{w}_{S^1(I_k)} \lesssim A_2$,
provided $\delta_0$ is sufficiently small depending on $A_1$.

Summing these bounds over all the intervals $I_k$, we obtain $\normo{w}_{S^1(I)} \le C(A_1, B, \delta_0)$,
which implies $$\normo{\jb{\nabla} w}_{L_{t,x}^\frac{2(d+2)}d(I \times \R^d)} +  \normo{w}_{  L_{t,x}^\frac{(d+2)(p-1)}2(I \times \R^d)} \le C(A_1, B, \delta_0).$$
This allows us to subdivide $I$ into $N_1 = C(B,
\delta_0)$ subintervals $J_k = [t_k, t_{k+1}]$ such that
$$\normo{\jb{\nabla} w}_{L_{t,x}^\frac{2(d+2)}d(J_k \times \R^d)} +  \normo{w}_{  L_{t,x}^\frac{(d+2)(p-1)}2(J_k \times \R^d)} \le \delta_0.$$
Choosing $\delta_1 = \delta_1(N_1,
A_1, A_2)$ sufficiently small, we apply Lemma \ref{le3.3v11} to
obtain for each $k$, and all $0 < \delta < \delta_1$,
\begin{align*}
\normo{u -w}_{L_{t,x}^\frac{2(d+2)}d \cap L_{t,x}^\frac{(d+2)(p-1)}2 ( J_k\times \R^d)} \le& C(k) \delta^\alpha,\\
\normo{u-w}_{S^1(J_k)} \le& C(k) (A_2 + \delta^\alpha),\\
\normo{u}_{S^1(J_k)} \le& C(k)(A_1 + A_2),\\
\normo{\jb{\nabla}((i\partial_t + \Delta)(u-w) +
e)}_{L_{t,x}^\frac{2(d+2)}{d+4}(J_k \times \R^d)} \le&
C(k)\delta^\alpha,
\end{align*}
provided we can show \eqref{eq1.34v11}, \eqref{eq1.35v11} hold with $t_0$ replaced by $t_k$. We verify this using an inductive argument.
We have
\begin{align*}
 \normo{u(t_{k+1}) - w(t_{k+1})}_{H^1} \lesssim &   \normo{u(t_0)-
w(t_0)}_{H^1} + \normo{\jb{\nabla}
e}_{L_{t,x}^\frac{2(d+2)}{d+4}([t_0, t_{k+1}]\times \R^d)}\\
&+ \normo{\jb{\nabla}((i\partial_t + \Delta)(u-w) + e)}_{L_{t,x}^\frac{2(d+2)}{d+4}([t_0, t_{k+1}]\times \R^d)}\\
\lesssim & A_2 + \delta + \sum\limits_{j = 0}^k C(j)\delta^\alpha,
\end{align*}
and
\begin{align*}
 &\normo{\jb{\nabla}e^{i(t-t_{k+1})\Delta} (u(t_{k+1}) -
 w(t_{k+1}))}_{L_{t,x}^\frac{2(d+2)}d}\\
\lesssim & \normo{\jb{\nabla} e^{i(t-t_0) \Delta}(u(t_0) -
w(t_0))}_{L_{t,x}^\frac{2(d+2)}d(I \times \R^d)}+ \normo{\jb{\nabla} e}_{L_{t,x}^\frac{2(d+2)}{d+4}}\\ & + \normo{\jb{\nabla}((i\partial_t + \Delta)(u-w) + e)}_{L_{t,x}^\frac{2(d+2)}{d+4}}\\
\lesssim & \delta + \sum\limits_{j= 0}^k C(j) \delta^\alpha.
\end{align*}
Here, $C(j)$ depends only on $j, A_1, A_2, \delta_0$.

Choosing $\delta_1$ sufficiently small depending on $N_1, A_1, A_2$, we can continue the inductive argument. This concludes the proof of Proposition \ref{pr4.3}.
\section{Linear Profile decomposition}
The linear profile decomposition was first established by H. Bahouri and P. G\'erard \cite{Bahouri-Gerard} for the energy critical wave equation in $\dot{H}^1(\R^d)$. Later, S. Keraani \cite{Keraani1} proved the linear profile decomposition for the energy critical Schr\"odinger equation in $\dot{H}^1(\R^d)$. At almost the same time, F. Merle and L. Vega \cite{Merle-Vega} gave the linear profile decomposition for the mass critical Schr\"odinger equation in $L^2(\R^d)$, then R. Carles
and S. Keraani \cite{Carles-Keraani} established this in $L^2(\R)$. In \cite{Keraani2}, S. Keraani use the $L^2$ profile decomposition to describe the minimal mass blowup solution.
Later, P. B\'egout and A. Vargas \cite{Begout-Vargas} extend these results to $L^2(\R^d),\  d \ge 3$.
In this section, we will give the linear profile decomposition in $H^1(\R^d)$ by using the linear profile decomposition in $L^2(\R^d)$. We first review the linear profile decomposition in $L^2(\R^d)$ in the following.
\begin{lem}[Profile decomposition in $L^2(\R^d)$, \cite{Begout-Vargas, Carles-Keraani, Merle-Vega}]\label{le8.1v14}
  Let $\{\phi_n\}_{n\ge 1}$ be a bounded sequence in $L^2(\R^d)$. Then up to passing to a subsequence of $\{\phi_n\}_{n\ge 1}$, there exists a sequence of functions
$\phi^j \in L^2(\R^d)$ and $( \theta_n^j, h_n^j, t_n^j, x_n^j, \xi_n^j)_{n\ge 1} \subset  \R/2\pi\mathbb{Z} \times (0,\infty)\times  \R \times \R^d \times \R^d$,
with
 \begin{equation}\label{eq8.4v14}
\tfrac{h_n^m}{h_n^j} + \tfrac{h_n^j}{h_n^m} + \tfrac{|t_n^j -
t_n^m|}{(h_n^j)^2}+ h_n^j | \xi_n^j -\xi_n^m|  +  \left|\tfrac{x_n^j- x_n^m}{h_n^j} +
\tfrac{2(t_n^j \xi_n^j - t_n^m \xi_n^m)}{h_n^j} \right|   \to
\infty, \text{ as }  n\to \infty, \text{  for }  j \ne m,
\end{equation}
where
\begin{align}
& h_n^j \to h_\infty^j\in \{0, 1, \infty\},\  h_n^j = 1 \text{ if }  h_\infty^j = 1, \label{eq8.1v14}\\
& \tau_n^j = - \tfrac{t_n^j}{(h_n^j)^2} \to \tau_\infty^j \in [-\infty, \infty], \text{ as } n\to \infty,\label{eq8.2v14}\\
 & \xi_n^j = 0 \text{  if  }  \quad \underset{ n \to \infty} {\limsup}\,|h_n^j \xi_n^j| < \infty, \label{eq8.3v14}
\end{align}
such that $\forall \,k \ge 1$, there exists $r^k_n\in L^2(\R^d)$,
\begin{equation}
\phi_n(x) = \sum\limits_{j=1}^k  T_n^j \phi^j(x) + r^k_n(x),
\end{equation}
here $T_n^j$ is defined by $T_n^j \phi(x) = e^{i\theta_n^j}
e^{ix\cdot\xi_n^j} e^{-it_n^j \Delta} \Big(\tfrac1{(h_n^j)^\frac{d}2
} \phi \big(\tfrac{\cdot - x_n^j }{h_n^j}\big)\Big)(x)$. The
remainder $r_n^k$ satisfies
\begin{equation}\label{eq8.6v14}
 \underset{ n\to \infty} \limsup \, \norm{e^{it\Delta} r_n^k}_{L_t^q L_x^r(\R\times \R^d)} \to 0, \text{ as }  k\to \infty,
\end{equation}
 where $(q,r)$ is $L^2$-admissible, and $2 < q < \infty \text{ when } d\ge 2,\  4 < q < \infty \text{ when } d = 1$.
 We also have as $n\to \infty$,
\begin{align}
& \ \normo{e^{it\Delta} T_n^j(\phi^j) e^{it\Delta} T_n^m(\phi^m)}_{L_{t,x}^\frac{d+2}d(\R \times \R^d)} \to 0, \  \jb{T_n^j(\phi^j), T_n^m(\phi^m)}_{L^2}\to 0,  \text{ for } j\ne m,\\
 & \text{ and } \forall \,1 \le j\le k, \jb{T_n^j (\phi^j), r_n^k}_{L^2} \to 0, \  (T_n^j)^{-1} r_n^k \rightharpoonup 0 \text{ in } L^2(\R^d).\label{eq8.8v14}
\end{align}
As a consequence, we have the mass decoupling property:
\begin{equation}\label{eq8.9v14}
\forall\, k \ge 1,\
\norm{\phi_n}_{L^2}^2 - \sum\limits_{j=1}^k \norm{\phi^j}_{L^2}^2 - \norm{r_n^k}_{L^2}^2 \to 0.
\end{equation}
\end{lem}
\begin{proof}
We only need to show \eqref{eq8.1v14}, \eqref{eq8.3v14}, \eqref{eq8.6v14} and \eqref{eq8.8v14}. Other statements in the theorem are stated in the profile decomposition in $L^2(\R^d)$ proved in \cite{Begout-Vargas, Carles-Keraani, Merle-Vega}. Without loss of generality, we assume that the sequence is up to a subsequence in the following.

To show \eqref{eq8.1v14}, we only need to prove that we may take $h_\infty^j$ and $h_n^j$ to be 1 when $h_\infty^j \in (0, \infty)$. In fact,
if $h_n^j \to h_\infty^j \in (0, \infty),\text{  as }  n\to \infty$, we have
\begin{align*}
& e^{i\theta_n^j} e^{ix\cdot\xi_n^j} e^{-it_n^j \Delta} \bigg(\tfrac1{(h_n^j)^\frac{d}2 } \phi^j \Big(\tfrac{\cdot - x_n^j }{h_n^j}\Big)\bigg)(x)\\
= & \  e^{i\theta_n^j} e^{ix\cdot\xi_n^j} e^{-it_n^j \Delta} \bigg(\tfrac1{(h_\infty^j)^\frac{d}2 } \phi^j \Big(\tfrac{\cdot - x_n^j }{h_\infty^j}\Big)\bigg)(x)\\
+ &\  e^{i\theta_n^j} e^{ix\cdot\xi_n^j} e^{-it_n^j \Delta}
\bigg(\tfrac1{(h_n^j)^\frac{d}2 } \phi^j \Big(\tfrac{\cdot - x_n^j
}{h_n^j}\Big)\bigg)(x)- e^{i\theta_n^j} e^{ix\cdot\xi_n^j}
e^{-it_n^j \Delta} \bigg(\tfrac1{(h_\infty^j)^\frac{d}2 } \phi^j
\Big(\tfrac{\cdot - x_n^j }{h_\infty^j}\Big)\bigg)(x).
\end{align*}
Note that
\begin{align*}
& \normo{e^{i\theta_n^j} e^{ix\cdot\xi_n^j} e^{-it_n^j \Delta} \bigg(\tfrac1{(h_n^j)^\frac{d}2 } \phi^j \Big(\tfrac{\cdot - x_n^j }{h_n^j}\Big)
\bigg)(x)- e^{i\theta_n^j} e^{ix\cdot\xi_n^j} e^{-it_n^j \Delta} \bigg(\tfrac1{(h_\infty^j)^\frac{d}2 } \phi^j \Big(\tfrac{\cdot - x_n^j }{h_\infty^j}\Big)\bigg)(x)}_{L^2}\\
= & \normo{\phi^j(x) -
\left(\tfrac{h_n^j}{h_\infty^j}\right)^{\frac{d}2}
\phi^j\Big(\tfrac{h_n^j}{h_\infty^j} x\Big)}_{L^2}\to 0, \text{ as }
n\to \infty.
\end{align*}
So we can put $e^{i\theta_n^j} e^{ix\cdot\xi_n^j} e^{-it_n^j \Delta}
\Big(\tfrac1{(h_n^j)^\frac{d}2 } \phi^j \big(\tfrac{\cdot - x_n^j
}{h_n^j}\big)\Big)(x)- e^{i\theta_n^j} e^{ix\cdot\xi_n^j} e^{-it_n^j
\Delta} \Big(\tfrac1{(h_\infty^j)^\frac{d}2 } \phi^j
\big(\tfrac{\cdot - x_n^j }{h_\infty^j}\big)\Big)(x)$ into the
remainder term, by the Strichartz estimate. We now shift $\phi^j(x)$ by
$\tfrac1{(h_\infty^j)^\frac{d}2} \phi^j(\tfrac{x}{h_\infty^j})$, and
$(\theta_n^j, h_n^j, t_n^j, x_n^j, \xi_n^j)$ by $(\theta_n^j,1,
t_n^j, x_n^j, \xi_n^j)$. It is easy to see that
\eqref{eq8.4v14}-\eqref{eq8.3v14} are not affected. Thus, we
conclude  \eqref{eq8.1v14}.

We now show \eqref{eq8.3v14}.
If $h_n^j \xi_n^j \to \xi^j\in \R^d, \text{ as }  n\to \infty, \text{ for some }  1 \le j \le k$. By the Galilean transform
\begin{equation*}
e^{it_0 \Delta}( e^{ix\cdot\xi_0} \phi(x)) = e^{i(x\cdot\xi_0 - t_0|\xi_0|^2)} e^{it_0 \Delta} \phi(x- 2t_0 \xi_0),
\end{equation*}
we have
\begin{align*}
& e^{i\theta_n^j} e^{ix\cdot\xi_n^j} e^{-it_n^j \Delta} \bigg(\tfrac1{(h_n^j)^\frac{d}2 } \phi^j \Big(\tfrac{\cdot - x_n^j }{h_n^j}\Big)\bigg)(x)\\
= & \tfrac1{(h_n^j)^\frac{d}2 } e^{i\theta_n^j}  e^{-it_n^j |\xi_n^j|^2}  e^{-it_n^j \Delta} \bigg( e^{ix\cdot\xi_n^j} \phi^j\Big(\tfrac{x-x_n^j - 2t_n^j \xi_n^j}{h_n^j}\Big) \bigg)\\
= & \tfrac1{(h_n^j)^\frac{d}2 } e^{i\theta_n^j}  e^{i(t_n^j |\xi_n^j|^2 + x_n^j\cdot \xi_n^j)}
e^{-it_n^j \Delta} \bigg(\Big(e^{i\langle\cdot,~ h_n^j \xi_n^j\rangle} \phi^j\Big)\Big(\tfrac{\cdot -x_n^j}{h_n^j}\Big) \bigg)(x- 2t_n^j \xi_n^j)\\
= & \tfrac1{(h_n^j)^\frac{d}2 } e^{i(\theta_n^j + t_n^j |\xi_n^j|^2 + x_n^j\cdot \xi_n^j)}
  e^{-it_n^j \Delta} \bigg(\Big(e^{i\langle\cdot,~ \xi^j\rangle} \phi^j\Big)\Big(\tfrac{\cdot -x_n^j}{h_n^j}\Big) \bigg)(x- 2t_n^j \xi_n^j)\\
&+   \tfrac1{(h_n^j)^\frac{d}2 } e^{i(\theta_n^j + t_n^j |\xi_n^j|^2 + x_n^j\cdot \xi_n^j)}
 e^{-it_n^j \Delta} \bigg(\Big(e^{i\langle\cdot,~ h_n^j \xi_n^j\rangle} \phi^j\Big)\Big(\tfrac{\cdot -x_n^j}{h_n^j}\Big) \bigg)(x- 2t_n^j \xi_n^j)\\
&- \tfrac1{(h_n^j)^\frac{d}2 } e^{i(\theta_n^j + t_n^j |\xi_n^j|^2 +
x_n^j\cdot \xi_n^j)}  e^{-it_n^j \Delta}
\bigg(\Big(e^{i\langle\cdot,~ \xi^j\rangle}
\phi^j\Big)\Big(\tfrac{\cdot -x_n^j}{h_n^j}\Big) \bigg)(x- 2t_n^j
\xi_n^j).
\end{align*}
We see
\begin{align*}
& \normo{\tfrac1{(h_n^j)^\frac{d}2 } e^{i(\theta_n^j + t_n^j |\xi_n^j|^2 + x_n^j\cdot \xi_n^j)}
 e^{-it_n^j \Delta} \bigg(\Big(e^{i\langle\cdot,~ h_n^j \xi_n^j\rangle} \phi^j\Big)\Big(\tfrac{\cdot -x_n^j}{h_n^j}\Big) -
  \Big(e^{i\langle\cdot,~ \xi^j\rangle} \phi^j\Big)\Big(\tfrac{\cdot -x_n^j}{h_n^j}\Big)\bigg)(x- 2t_n^j \xi_n^j)}_{L^2}\\
= & \normo{(e^{ix\cdot h_n^j \xi_n^j} - e^{ix\cdot\xi^j}) \phi^j(x)}_{L^2} \to 0, \text{ as } n \to \infty.
\end{align*}
We now replace
$( \theta_n^j, h_n^j, t_n^j, x_n^j, \xi_n^j)$ by $(\theta_n^j + t_n^j |\xi_n^j|^2 + x_n^j\cdot \xi_n^j, h_n^j, t_n^j, x_n^j + 2t_n^j \xi_n^j, 0)$, and $\phi^j(x)$ by $e^{ix\cdot\xi^j} \phi^j(x)$, we can verify that \eqref{eq8.4v14}-\eqref{eq8.3v14} are not affected.
So we can take $\xi_n^j = 0$ when $\underset{n\to \infty} \limsup \, |h_n^j \xi_n^j | < \infty$.

For the profile decomposition in \cite{Begout-Vargas, Carles-Keraani, Merle-Vega}, the remainder $r_n^k$ satisfies
\begin{equation*}
\underset{n\to \infty} \limsup \, \normo{e^{it\Delta} r_n^k}_{L_{t,x}^\frac{2(d+2)}d(\R \times \R^d)} \to 0, \text{ as }  k\to \infty.
 \end{equation*}
By using interpolation, the Strichartz estimate and \eqref{eq8.9v14}, we easily obtain \eqref{eq8.6v14}.

To show \eqref{eq8.8v14}, we see in \cite{Begout-Vargas, Carles-Keraani, Merle-Vega}, we already have
\begin{equation}\label{eq8.10v14}
(T_n^j)^{-1} r_n^j \rightharpoonup 0 \text{ in } L^2(\R^d),\text{ as }  n\to \infty.
\end{equation}
Since
$r_n^k(x) = r_n^j(x) - \sum\limits_{ m = j+ 1}^k T_n^m(\phi^m)(x), \text{ when } 1 \le j < k$,
so by \eqref{eq8.4v14} and \eqref{eq8.10v14},
\begin{equation*}
(T_n^j)^{-1} r_n^k = (T_n^j)^{-1} r_n^j - \sum\limits_{m = j+1}^k (T_n^j)^{-1} T_n^m \phi^m \rightharpoonup 0 \text{ in }  L^2(\R^d), \text{ as }  n\to \infty.
\end{equation*}
\end{proof}
We can now show the linear profile decomposition in $H^1(\R^d)$.
\begin{thm}\label{le5.2}
  Let $\{\varphi_n\}$ be a bounded sequence in $H^1(\R^d)$. Then up to passing to a subsequence of $\{\varphi_n\}$, there exists a sequence of functions $\varphi^j \in H^1(\R^d)$ and $( \theta_n^j, t_n^j, x_n^j)_{n\ge 1} \subset  \R/2\pi\mathbb{Z}  \times \R \times \R^d $, with when $n\to \infty$,
  \begin{align}
  & {|t_n^j - t_n^m|} +  \left|{x_n^j - x_n^m}\right| \to \infty,\  \forall\, j \ne m,    \label{eq2.13v3new}
\end{align}
such that for any $k \in \N$, there exists $w_n^k \in H^1(\R^d)$,
\begin{equation}\label{eq2.18v3new}
\begin{split}
e^{it\De} \varphi_n & = \sum\limits_{j =1 }^k e^{i\theta_n^j} e^{i(t-t_n^j)\De}  \varphi^j(x-x_n^j) + e^{it\De} w_n^k.
\end{split}
\end{equation}
The remainder $w_n^k$ satisfies
\begin{equation}\label{eq2.15v3new}
\underset{n\to \infty}{\limsup}\  \norm{\jb{\na} e^{it\De} w_n^k}_{L_t^q L_x^r(\R \times \R^d)} \to 0, \text{ as }  k \to \infty,
\end{equation}
where $(q,r)$ is $L^2$-admissible, and $2 < q < \infty \text{ when } d\ge 2,\, 4 < q < \infty \text{ when }  d = 1$.
Moreover, we have the following decoupling properties: $\forall \, k\in \N$,
%
\begin{align}
& \big\||\na|^s \varphi_n\big\|_{L^2}^2 - \sum\limits_{j=1}^k \normo{|\na|^s  e^{-i t_n^j\De} \varphi^j }_{L^2}^2 - \big\||\na|^s w_n^k\big\|_{L^2}^2 \to 0,  \ s = 0,1, \label{eq2.21v3new} \\
& \E(\varphi_n) - \sum\limits_{j=1}^k \E(e^{-i t_n^j\De} \varphi^j ) - \E(w_n^k) \to 0, \label{eq2.22v3new}  \\
& \cS_\om(\varphi_n) - \sum\limits_{j=1}^k \cS_\om(e^{-i t_n^j\De} \varphi^j) - \cS_\om ( w_n^k) \to 0, \label{eq2.23v3new}\\
& \K(\varphi_n) - \sum\limits_{j=1}^k \K(e^{-i t_n^j\De} \varphi^j) - \K(w_n^k) \to 0,  \label{eq2.25v3new} \\
& \cH_\om(\varphi_n) - \sum\limits_{j=1}^k \cH_\om(e^{-i t_n^j\De} \varphi^j) - \cH_\om(w_n^k) \to 0, \text{  as } n\to \infty.  \label{eq2.24v3new}
\end{align}
\end{thm}
\begin{proof}
We divide the proof into four steps.

{\it Step 1. }  Applying Lemma \ref{le8.1v14} to $\{ \jb{\nabla} \varphi_n\}$, we have
\begin{equation}\label{eq4.22v6new}
\jb{\nabla} \varphi_n(x) = \sum\limits_{j=1}^k e^{i\theta_n^j} e^{ix\cdot\xi_n^j} e^{-it_n^j \Delta}
 \bigg(\tfrac1{(h_n^j)^\frac{d}2} \Big(\jb{\nabla} \varphi^j\Big)\Big(\tfrac{\cdot-x_n^j} {h_n^j}\Big)\bigg)(x) + \jb{\nabla} w_n^k(x),
\end{equation}
where
\begin{align}
 & \tau_n^j = - \tfrac{t_n^j}{(h_n^j)^2} \to \tau_\infty^j \in [-\infty, \infty],  \nonumber \\
 &   h_n^j \to h_\infty^j \in \{0, 1, \infty\},  \text{ as } n\to \infty, \text{ and }  h_n^j = 1 \text{ if  } h_\infty^j = 1,  \label{eq4.22v7new} \\
 & \xi_n^j = 0 \text{  if  }  \quad \underset{ n \to \infty} {\limsup}\,|h_n^j \xi_n^j| < \infty, \label{eq8.24v14}\\
 & \tfrac{h_n^j}{h_n^m} + \tfrac{h_n^m}{h_n^j}+ h_n^j | \xi_n^j -\xi_n^m|+ \tfrac{|t_n^j - t_n^m|}{(h_n^j)^2} +  \left|\tfrac{x_n^j - x_n^m}{h_n^j} + \tfrac{2(t_n^j \xi_n^j - t_n^m \xi_n^m)}{h_n^j}\right| \to \infty, \text{ as }  n\to \infty, \text{ for }  j \ne m, \nonumber \\
& \underset{n\to\infty} \limsup \norm{ \jb{\nabla} e^{it\Delta}  w_n^k}_{L_t^q L_x^r(\R \times \R^d)} \to 0, \text{ as } k\to \infty, \nonumber\\
& \text{where } (q,r) \text{ is $L^2$-admissible and $2 < q < \infty$ when $d\ge 2$, $4 < q < \infty$ when $d = 1$}. \nonumber
\end{align}
Moreover, we have
\begin{align}
\forall\  k\ge 1,
\quad \norm{\jb{\nabla} \varphi_n}_{L^2}^2 - \sum\limits_{j=1}^k \normo{\jb{\nabla}  \varphi^j}_{L^2}^2 - \norm{\jb{\nabla} w_n^k}_{L^2}^2 \to 0, \text{ as }  n\to \infty. \nonumber
\end{align}
By \eqref{eq4.22v6new} and the Galilean transform, we have
\begin{align}\nonumber
e^{it\Delta}\varphi_n(x) & =  \sum\limits_{j=1}^k e^{it\Delta}
\jb{\nabla}^{-1}e^{i\theta_n^j} e^{ix\cdot \xi_n^j} e^{-i {t_n^j}
\Delta} \bigg( \tfrac1{(h_n^j)^\frac{d}2} \Big( \jb{\nabla}
\varphi^j\Big)\Big(\tfrac{\cdot- x_n^j}{h_n^j}\Big) \bigg)\big(
x\big)  + e^{it\Delta} w_n^k(x) \\\label{equ4.22}
& = \sum\limits_{j=1}^k \jb{\nabla}^{-1} {G}_n^j(e^{it\Delta} \jb{\nabla} \varphi^j)(x) + e^{it\Delta} w_n^k(x),
\end{align}
where ${G}_n^j(e^{it\Delta} \varphi)(x) =  e^{it\Delta} T_n^j \varphi(x)$ and $T_n^j$ is defined in Lemma \ref{le8.1v14}.

By the density of $C_0^\infty(\R^d)$ in $H^1(\R^d)$, we can assume that $\varphi^j$ is smooth for all $j\geq1$.

{\it Step 2. } We now show $h_\infty^j =1$ and $\xi_n^j = 0$. By \eqref{eq4.22v7new} and \eqref{eq8.24v14}, we only need to exclude the case $h_\infty^j = 0, \infty$ and
$\underset{ n \to \infty} \limsup \, |h_n^j \xi_n^j| = \infty$.

By \eqref{equ4.22}, we have
\begin{equation*}
\jb{\nabla}^{-1} {G}_n^j(e^{it\Delta} \jb{\nabla} \varphi^j) = e^{it\Delta}(w_n^{j-1} - w_n^j),
\end{equation*}
which implies that
\begin{align}\label{eq2.26v3new}
\begin{split}
 \varphi^j & = \jb{\nabla}^{-1} e^{-it\Delta} ({G}_n^j)^{-1} \jb{\nabla} e^{it\Delta}( w_n^{j-1} - w_n^j)\\
                       &  = \jb{\nabla}^{-1}  (T_n^j)^{-1}   \jb{\nabla} ( w_n^{j-1} - w_n^j).
\end{split}
\end{align}
We note by \eqref{eq8.8v14}, for $1 \le m \le j$,
\begin{equation}\label{eq8.28v14}
(T_n^m)^{-1} \jb{\nabla}w_n^j \rightharpoonup 0 \text{ in } L^2(\R^d), \text{ as } n\to \infty,
\end{equation}
so
\begin{equation*}
 \jb{\nabla}^{-1}  (T_n^m)^{-1} \jb{\nabla} w_n^j \rightharpoonup 0 \text{ in }  H^1(\R^d), \text{ as } n\to \infty.
 \end{equation*}
This together with \eqref{eq2.26v3new} yields
\begin{equation}\label{eq2.28v3new}
 \jb{\nabla}^{-1}  (T_n^j )^{-1}   \jb{\nabla}   w_n^{j-1}  \rightharpoonup  \varphi^j \text{ in }  H^1(\R^d), \text{ as } n\to \infty.
\end{equation}
 Meanwhile, since $\{w_n^{j-1}\}$ is bounded in $H^1(\R^d)$,  there exists a subsequence of $\{w_n^{j-1}\}$ and $\psi^j \in H^1(\R^d)$ that
\begin{equation}\label{eq2.29v3new}
w_n^{j-1} \rightharpoonup \psi^j \text{ in }  H^1(\R^d), \text{ as }  n \to \infty.
\end{equation}
  Combining \eqref{eq2.28v3new}, \eqref{eq2.29v3new}, we get
  \begin{equation*} \label{eq2.30v3new}
\jb{\nabla}^{-1}(T_n^j)^{-1} \jb{\nabla} \psi^j \rightharpoonup \varphi^j  \text{ in }  H^1(\R^d), \text{ as } n\to \infty.
\end{equation*}
Then by the Rellich-Kondrashov theorem(see \cite{Lieb-Loss}), for any ball $B_K$ centered at the origin with radius $K$, we have
\begin{equation*} \label{eq2.30v3new}
\jb{\nabla}^{-1}(T_n^j)^{-1} \jb{\nabla} \psi^j \to \varphi^j  \text{ in }  L^q(B_K), \text{ as } n\to \infty
\end{equation*}
for $2 \le q < \infty$ for $d=1,2$ and $2\leq q<\frac{2d}{d-2}$ for $d\geq3$.
So
\begin{align}\nonumber
\normo{\varphi^j}_{L^q(B_K)}
& \le \underset{n\to \infty} \liminf\, \normo{\jb{\nabla}^{-1} (T_n^j)^{-1} \jb{\nabla} \psi^j}_{L^q(\R^d)}\\\nonumber
& =  \underset{n\to \infty} \liminf\, (h_n^j)^\frac{d}2 \normo{
\left(\tfrac1{\jb{h_n^j \nabla}} e^{it_n^j \Delta}
e^{-i\langle\cdot,~ \xi_n^j\rangle} \jb{\nabla} \psi^j\right)(h_n^j
x)}_{L^q(\R^d)}\\\nonumber & = \underset{n\to \infty} \liminf\,
(h_n^j)^{\frac{d}2-\frac{d}q} \normo{\tfrac1{\jb{h_n^j \nabla}}
e^{it_n^j \Delta} e^{-ix\cdot\xi_n^j} \jb{\nabla}
\psi^j}_{L^q(\R^d)}\\\nonumber & \lesssim \underset{n\to \infty}
\liminf\,  \normo{\tfrac{|h_n^j
\nabla|^{\frac{d}2-\frac{d}q}}{\jb{h_n^j \nabla}} e^{it_n^j \Delta}
e^{-ix\cdot\xi_n^j} \jb{\nabla} \psi^j}_{L^2(\R^d)}\\\label{equ4.23}
& =  \underset{n\to \infty} \liminf\,\normo{\tfrac{|h_n^j \xi|^{\frac{d}2-\frac{d}q}}{\jb{h_n^j \xi}} \jb{\xi + \xi_n^j} \widehat{\psi^j}(\xi + \xi_n^j)}_{L^2(\R^d)}.
\end{align}

{\bf Case I.}  $|h_n^j \xi_n^j|\to \infty, \text{ as } n\to \infty$.

When $h_\infty^j<\infty$, we take $q = 2$, then by the dominated convergence theorem, we have
\begin{align*}
 \normo{\tfrac{1} {\jb{h_n^j \xi}} \jb{\xi + \xi_n^j} \widehat{\psi^j}(\xi + \xi_n^j)}_{L^2(\R^d)}= \normo{\tfrac{1}{\jb{h_n^j (\xi - \xi_n^j)}} \jb{\xi } \widehat{\psi^j}(\xi )}_{L^2(\R^d)} \to 0, \text{ as } n \to \infty.
\end{align*}
This implies $\varphi^j=0$ in $H^1(\R^d)$.

When $h_\infty^j=\infty$, we take $q = 2$, if $\xi_n^j \to \bar{\xi}\in \R^d$, for any small number $\varepsilon>0$, we can take $r=r(\varepsilon)>0$ such that
$$\normo{\tfrac{1}{\jb{h_n^j (\xi - \xi_n^j)}} \jb{\xi } \widehat{\psi^j}(\xi )}_{L^2(B_r(\bar{\xi}))}\leq\varepsilon$$
by the continuity of the integral. And we use the dominated convergence theorem again to obtain
$$\normo{\tfrac{1}{\jb{h_n^j (\xi - \xi_n^j)}} \jb{\xi } \widehat{\psi^j}(\xi )}_{L^2(\R^d\backslash B_r(\bar{\xi}))}\to 0.$$
Since $\varepsilon$ is arbitrary, we have $\varphi^j=0$ in $H^1(\R^d)$. If $|\xi_n^j| \to \infty$, the dominated convergence theorem will guarantee
that $\varphi^j=0$ in $H^1(\R^d)$.

{\bf Case II.} If $\underset{n \to \infty} \limsup \, |h_n^j \xi_n^j| < \infty$, then we obtain $\xi_n^j =0$ by \eqref{eq8.24v14}.

When $h^j_\infty=0$, we can take some $q>2$, we have
\begin{align*}
\normo{\tfrac{|h_n^j \xi|^{\frac{d}2-\frac{d}q}}{\jb{h_n^j \xi}}
\jb{\xi } \widehat{\psi^j}(\xi )}_{L^2(\R^d)}\to 0, \text{ as } n\to
\infty,
\end{align*}
so $\normo{\varphi^j}_{L^q(B_K)} = 0$ for any $K>0$. Therefore, $\normo{\varphi^j}_{L^q(\R^d)} = 0$. Since $\varphi^j$ is smooth, $\varphi^j=0$ in $H^1(\R^d)$.

When $h^j_\infty=\infty$, we take $q=2$. Then by the dominated convergence theorem, we have
\begin{equation*}
 \normo{\tfrac{1}{\jb{h_n^j \xi}} \jb{\xi }
 \widehat{\psi^j}(\xi )}_{L^2(\R^d)} \to 0, \text{ as } n \to \infty,
\end{equation*}
which implies $\varphi^j=0$ in $H^1(\R^d)$.

Combining the above facts, we conclude that $h^j_\infty=1$ and $\xi_n^j=0$. This together with \eqref{equ4.22} implies  \eqref{eq2.13v3new} and  \eqref{eq2.18v3new}.

{\it Step 3.} We now turn to \eqref{eq2.21v3new}.
By \eqref{eq2.18v3new}, we have
\begin{equation*}
\varphi_n(x) = \sum\limits_{j=1}^k e^{i\theta_n^j} e^{-i{t_n^j}\Delta} \varphi^j(x-x_n^j) + w_n^k(x),
\end{equation*}
so we only need to show the orthogonality:
\begin{align*}
&  \left\langle \mu\jb{\nabla} (e^{i\theta_n^j} e^{- i t_n^j  \Delta} \varphi^j(x-x_n^j) ), \mu \jb{\nabla}  (e^{i\theta_n^m} e^{-i t_n^m \Delta} \varphi^m(x-x_n^m) )\right\rangle_{L^2} \to 0, \ \forall\, j \ne m,\\
& \left\langle\mu \jb{\nabla} (e^{i\theta_n^j} e^{- i t_n^j  \Delta} \varphi^j)(x-x_n^j), \mu \jb{\nabla}w_n^k(x)\right\rangle_{L^2} \to 0, \text{ for } 1 \le j \le k, \text{ as }  n\to \infty,
\end{align*}
for $\mu = \frac1{\jb{\nabla}}$ and $\mu= \frac{|\nabla|}{\jb{\nabla}}$, respectively. This follows from
\begin{align*}
& \left\langle \mu\jb{\nabla} (e^{i\theta_n^j} e^{-i t_n^j  \Delta} \varphi^j(x-x_n^j) ), \mu \jb{\nabla}  (e^{i\theta_n^m} e^{-i t_n^m \Delta} \varphi^m )(x-x_n^m) \right\rangle_{L^2}\\
= & \bigg\langle  S_n^{m,j} \mu \jb{\nabla} \varphi^j,   \mu \jb{\nabla}  \varphi^m\bigg\rangle_{L^2} \to 0, \text{ as } n\to \infty,
\end{align*}
where
\begin{align*}
S_n^{m,j} \varphi(x)&  = e^{i(\theta_n^j -\theta_n^m)}   e^{-i(t_n^j - t_n^m)\Delta}    \varphi ( x-(x_n^j - x_n^m))
\rightharpoonup 0, \text{ as }  n\to \infty,\  \forall \, \varphi\in L^2,
\text{ by } \eqref{eq2.13v3new}
\end{align*}
and
\begin{align*}
 \left\langle\mu \jb{\nabla}(e^{i\theta_n^j} e^{-i{t_n^j}  \Delta} \varphi^j(x-x_n^j)), \mu \jb{\nabla}w_n^k(x)\right\rangle_{L^2}
= &  \left\langle    \mu \jb{\nabla} e^{i\theta_n^j} e^{-i{t_n^j}  \Delta} \varphi^j(x-x_n^j), \mu \jb{\nabla}w_n^k(x)\right\rangle_{L^2} \\
= &  \left\langle   \mu^2 \jb{\nabla} \varphi^j, ( T_n^j )^{-1} \jb{\nabla}w_n^k\right\rangle_{L^2}\\
 \to & \ 0, \text{ as } n \to \infty,\  \forall \, 1 \le j \le k, \text{ by \eqref{eq8.28v14}.}
\end{align*}

{\it Step 4.} Since we already have \eqref{eq2.21v3new}, to obtain \eqref{eq2.22v3new}, \eqref{eq2.23v3new}, \eqref{eq2.25v3new}, \eqref{eq2.24v3new}, it suffices to show
\begin{equation} \label{eq2.32v3new}
\normo{\varphi_n}_{L^{p+1}}^{p+1} - \sum\limits_{j=1}^k \normo{ e^{-it_n^j\Delta}  \varphi^j}_{L^{p+1}}^{p+1} - \normo{w_n^k}_{L^{p+1}}^{p+1} \to 0, \text{ as } n\to  \infty,
\text{ for }  1 + \frac4d \le p \le 1 + \frac4{d-2}.
\end{equation}

Suppose that $t_\infty^1  \in \R$, then the refined Fatou Lemma(see \cite{Killip-Visan, Lieb-Loss}) shows that
\begin{equation}
\underset{n\to \infty} \lim \left | \normo{\varphi_n}_{L^{p+1}}^{p+1} - \normo{e^{i\theta_n^1} e^{-it_n^1  \Delta} \varphi^1}_{L^{p+1}}^{p+1} - \normo{w_n^1}_{L^{p+1}}^{p+1}  \right|  = 0.
\end{equation}
  Next, suppose that $t_\infty^1 =  \pm \infty$, then by the dispersive estimate, we obtain
\begin{align*}
& \underset{n\to \infty} \lim \left| \normo{\varphi_n}_{L^{p+1}}^{p+1} - \normo{e^{i\theta_n^1} e^{-it_n^1  \Delta} \varphi^1}_{L^{p+1}}^{p+1} - \normo{w_n^1}_{L^{p+1}}^{p+1}  \right|\nonumber\\
\le & \underset{n\to \infty} \lim  \left | \normo{
 \varphi_n}_{L^{p+1}}^{p+1} - \normo{w_n^1}_{L^{p+1}}^{p+1} \right| +   \underset{n\to \infty} \lim  \normo{ e^{-it_n^1 \Delta} \varphi^1}_{L^{p+1}}^{p+1}\nonumber \\
\le & \underset{n\to \infty} \lim \left( \normo{\varphi_n}_{L^{p+1}}^p
+ \normo{   w_n^1}_{L^{p+1}}^p\right) \normo{ e^{-it_n^1 \Delta} \varphi^1}_{L^{p+1}} +    \underset{n\to \infty} \lim  \normo{ e^{-it_n^1  \Delta} \varphi^1}_{L^{p+1}}^{p+1}. \label{eq8.34v14}\\
\lesssim & \underset{n\to \infty} \lim \left( \normo{ \varphi_n}_{H^1}^p
+ \normo{   w_n^1}_{H^1}^p\right) \normo{ e^{-it_n^1 \Delta} \varphi^1}_{L^{p+1}}
  +  \underset{n\to \infty} \lim \normo{ e^{-it_n^1 \Delta }\varphi^1}_{L^{p+1}}^{p+1}
 = 0.
\end{align*}
Thus, we have proved
 \begin{equation}\label{eq2.40v3new}
 \normo{\varphi_n}_{L^{p+1}}^{p+1} - \normo{ e^{-it_n^1  \Delta} \varphi^1}_{L^{p+1}}^{p+1} - \normo{w_n^1}_{L^{p+1}}^{p+1}\to 0, \text{ as } n \to \infty.
\end{equation}
 Similarly, we can show
\begin{equation}
\normo{w_n^1}_{L^{p+1}}^{p+1} - \normo{ e^{-it_n^2  \Delta} \varphi^2}_{L^{p+1}}^{p+1} - \normo{w_n^2}_{L^{p+1}}^{p+1} \to 0, \text{ as  } n \to \infty,
\end{equation}
which together with \eqref{eq2.40v3new} shows \eqref{eq2.32v3new} when $k = 2$.
Repeating this procedure, we obtain \eqref{eq2.32v3new} for any $k\ge 1$.
\end{proof}

\section{Extraction of a critical element}
In this section, we show the existence of the critical element in
the general case by using the profile decomposition and the
long-time perturbation theory.

By Proposition \ref{th4.1}({\it v}), it suffices for Theorem \ref{th1.4} to show that any solution $u$ to \eqref{eq1.1} with $u_0 \in A_{\omega, +}$ satisfies
\begin{equation*}
\|u\|_{L_{t,x}^\frac{2(d+2)}d \cap L_{t,x}^\frac{(d+2)(p-1)}2 ( I_{max} \times \R^d)} < \infty,
\end{equation*}
where $I_{max}$ denotes the maximal interval where $u$ exists.

To this end, for $m > 0$, let
\begin{equation}\label{eq5.1}
\Lambda_\omega(m) = \sup\Big\{ \|u\|_{K( I_{max})}:
\begin{array}{l}
\text{ $u$ is a solution to \eqref{eq1.1} with  }\\
 \text{ $u_0 \in A_{\omega, +}$ and $\cS_\om(u) \le m$}
\end{array} \Big\}
\end{equation}
with $\|u\|_{K(I)}:=\|u\|_{L_{t,x}^\frac{2(d+2)}d \cap
L_{t,x}^\frac{(d+2)(p-1)}2 (I\times\R^d)}$and define
\begin{equation} \label{eq5.2}
m_\om^* = \sup\{ m>0: \Lambda_\omega(m)< \infty\}.
\end{equation}
 If  $u_0\in A_{\om, +}$ with  $\cS_\om(u_0) \le m$ sufficiently small, then Lemma \ref{le2.8} shows
$\|u_0\|_{H^1} \ll 1$. Hence, Proposition \ref{th4.1}{(\textit{i})}
gives the finiteness of $\Lambda_\om(m)$, which implies $m_\om^* >
0$.

  Now our aim is to show $m_\om^* \ge  m_\om$ defined by \eqref{eq1.7}.
Suppose by contradiction that $m_\om^* < m_\om$, we will show the existence of the critical element.
In fact, by the definition of $m_\om^*$, we can take a sequence
$\{u_n\}$ of solutions (up to time translations) to \eqref{eq1.1}
that
\begin{align}
& u_n(t) \in A_{\om,+}, \text{ for $t\in I_n$,}
 \quad \text{ and }\cS_\om(u_n) \to m_\om^*, \text{   as $n\to \infty$,} \label{eq5.4}\\
& \underset{n\to \infty} {\lim} \|u_n\|_{K([0,\sup I_n))} =
\underset{n\to \infty} {\lim} \|u_n\|_{K((\inf I_n, 0])} =
\infty,\label{eq5.5}
\end{align}
where $I_n$ denotes the maximal interval of $u_n$ including 0.

By Lemma \ref{le2.8},
\begin{equation}\label{eq5.6}
\underset{n}{\sup}\  \norm{u_n}_{L_t^\infty H_x^1(I_n\times \R^d)}^2 \lesssim m_\om + \frac{m_\om}\om.
\end{equation}
Applying Theorem \ref{le5.2} to $\{u_n(0,x)\}$ and obtain some subsequence of $\{u_n(0,x)\}$(still denoted by the same symbol), then
there exists $ \varphi^j \in H^1(\R^d)$ and $(  \theta_n^j,  t_n^j, x_n^j)_{n\ge 1}$ of sequences in $\R/2\pi \mathbb{Z}   \times \R \times \R^d  $, with
  \begin{align}
 &  t_n^j   \to t_\infty^j \in [-\infty, \infty], \nonumber \\
  & {|t_n^j - t_n^m|}+ \left|{x_n^j-x_n^m}  \right|
  \to \infty,\   \forall \,j \ne m, \text{  as } n\to \infty,\label{eq6.48v3}
\end{align}
such that, $\forall \,k \in \N$, there exists $w_n^k\in H^1(\R^d)$,
\begin{align}
e^{it\De} u_n(0,x)  & = \sum\limits_{j=1}^k e^{i\theta_n^j} e^{i(t-t_n^j)\Delta}\varphi^j(x-x_n^j) + e^{it\De} w_n^k(x).\label{eq6.51v3}
\end{align}
The remainder $w_n^k$ satisfies
\begin{equation}\label{eq6.50v3}
\underset{n\to \infty}{\limsup}\  \norm{\jb{\na} e^{it\De} w_n^k}_{L_t^q L_x^r(\R \times \R^d)} \to 0, \text{ as }  k \to \infty,
\end{equation}
where $(q,r)$ is $L^2$-admissible, $2< q < \infty$ when $d\ge 2$ and $4 < q < \infty$ when $d = 1$.
Moreover, for any $k\in \N$, $s = 0,1$, we have
\begin{align}
& \big\||\na|^s u_n(0)\big\|_{L^2}^2 - \sum\limits_{j=1}^k \big\||\na|^s e^{-it_n^j\Delta} \varphi^j \big\|_{L^2}^2 -
\big\||\na|^s w_n^k\big\|_{L^2}^2 \to 0,\label{eq6.52v3}\\
& \E(u_n(0)) - \sum\limits_{j=1}^k \E(e^{-it_n^j\Delta} \varphi^j) - \E(w_n^k) \to 0,\\
& \cS_\om(u_n(0)) - \sum\limits_{j=1}^k \cS_\om(e^{-it_n^j\Delta} \varphi^j) - \cS_\om ( w_n^k) \to 0,\label{eq6.54v3}\\
& \K(u_n(0)) - \sum\limits_{j=1}^k \K(e^{-it_n^j\Delta} \varphi^j) - \K(w_n^k) \to 0,\\
& \cH_\om(u_n(0)) - \sum\limits_{j=1}^k \cH_\om(e^{-it_n^j\Delta} \varphi^j) - \cH_\om(w_n^k) \to 0, \label{eq6.55v3} \text{  as } n\to \infty.
\end{align}
Using Strichartz estimate, \eqref{eq6.52v3} and \eqref{eq5.6}, we
get
\begin{equation}\label{eq6.57v3}
\underset{k\in \mathbb{N}} \sup \ \underset{n\to \infty} \limsup\, \norm{ e^{it\Delta} w_n^k}_{ S^1(\R)}
\lesssim \underset{k \in \mathbb{N}}\sup\   \underset{n \to \infty} \limsup \,\norm{w_n^k}_{H^1} < \infty.
\end{equation}
Next, we construct the nonlinear profile. We define the nonlinear
profile $u^j \in C((T_{min}^j, T_{max}^j), H^1(\R^d))$ to be the
maximal lifespan solution of $i\partial_t u  + \Delta u =
|u|^\frac4d u - |u|^{p-1} u$  such that
\begin{equation} \label{eq6.65v3}
\normo{u^j(- t_n^j) - e^{-it_n^j\Delta} \varphi^j}_{H^1} \to 0, \text{ as  } n\to \infty.
\end{equation}
 The unique existence of $u^j$ around $t = -t_\infty^j$ is known in all cases, including $t_\infty^j = \pm \infty$(the latter corresponding to the existence of the wave operators), by using the standard iteration with the Strichartz estimate.

Let
\begin{equation}
u_n^j = e^{i\theta_n^j} u^j(t-t_n^j, x-x_n^j),
\end{equation}
then, the lifespan of $u_n^j$ is $( T_{min}^j + t_n^j,  T_{max}^j +
t_n^j)$.

For the linear profile decomposition \eqref{eq6.51v3},
we can give the corresponding nonlinear profile decomposition
\begin{equation} \label{eq5.21v12}
u_n^{< k}(t) = \sum\limits_{j=1}^k u_n^j(t).
\end{equation}
We will show $u_n^{< k} + e^{it\Delta} w_n^k$ is a good
approximation for $u_n$ provided that each nonlinear profile has
finite global Strichartz norm, which is the key to show the
existence of the critical element.

When $j$ large enough, we have the following basic fact about $u^j$:
\begin{lem}\label{le6.8v3}
  There exists $j_0\in \N$ such that $T_{min}^j = -\infty,\  T_{max}^j = \infty$ for $j > j_0$ and
\begin{equation}
\sum\limits_{j> j_0} \norm{ u^j}_{S^1(\R)}^2  \lesssim \sum\limits_{j > j_0} \norm{ \varphi^j}_{H^1}^2 < \infty.
\end{equation}
\end{lem}
\begin{proof}
By \eqref{eq5.6} and \eqref{eq6.52v3}, we have
\begin{equation*}
\sum\limits_{j = 1}^\infty \normo{\jb{\nabla}(  e^{-it_n^j \Delta}
\varphi^j)}_{L^2}^2 < \infty,
\end{equation*}
which
shows $\sum\limits_{j=1}^\infty \norm{\varphi^j}_{H^1}^2 < \infty$, and therefore
$\norm{\varphi^j}_{H^1} \to 0,
\text{ as } j \to \infty.$
By the small data global wellposedness and scattering theory together with \eqref{eq6.65v3}, we obtain when $j$ large enough, $\norm{ u^j}_{S^1(\R)} \lesssim \norm{\varphi^j}_{H^1}$ and so we have the desired result.
\end{proof}

\begin{lem}\label{le6.9v3}
In the nonlinear profile decomposition \eqref{eq5.21v12}, if
  \begin{equation} \label{eq6.92v3}
  \norm{ u^j}_{L_{t,x}^\frac{2(d+2)}d \cap L_{t,x}^\frac{(d+2)(p-1)}2((T_{min}^j, T_{max}^j)\times \R^d)} < \infty \text{ for }  1 \le j \le k,
  \end{equation}
then, we have $T_{min}^j = -\infty,\  T_{max}^j = \infty$, and
\begin{align}\label{eq6.93v3}
 \norm{ u_n^j}_{S^1(\R)}=\norm{ u^j}_{S^1(\R)} \lesssim 1,
\end{align}
for $1 \le j \le k$,
and there exists $B >  0$ such that
\begin{align}\label{eq6.94v3}
\underset{n\to \infty} \limsup\  \Big(\normo{
u_n^{<k}}_{L_{t,x}^\frac{(d+2)(p-1)}2 (\R \times \R^d)}
+\normo{\jb{\nabla} u_n^{< k} }_{L_t^\infty L_x^2\cap
L_{t,x}^\frac{2(d+2)}d(\R\times \R^d)} \Big) \le B.
\end{align}
\end{lem}
\begin{proof}
By Proposition \ref{th4.1}({\it v}) and \eqref{eq6.92v3}, we have
$T_{min}^j = -\infty,\  T_{max}^j = \infty$.

Using the Strichartz estimate and \eqref{eq6.92v3}, we get
\begin{equation*}
\norm{u^j}_{S^1(\R)} \lesssim 1, \text{ for }  1\le j \le k,
\end{equation*}
which implies (\ref{eq6.93v3}).

  We now turn to \eqref{eq6.94v3}.
By
\begin{align*}
\left|\Big|\sum\limits_{1\le j \le k} u_n^j \Big|^q - \sum\limits_{1\le j \le k} \Big|u_n^j \Big|^q \right| \le C_{k,q}
\sum\limits_{1\le j\ne m \le k} |u_n^j|^{q-1} |u_n^m|, \ 1 < q < \infty,
\end{align*}
we have
\begin{align}
\normo{\sum\limits_{1 \le j \le k} u_n^j}_{L_{t,x}^\frac{(d+2)(p-1)}2(\R\times \R^d)}^\frac{(d+2)(p-1)}2
\le  & \sum\limits_{1\le j \le k} \normo{u_n^j}_{L_{t,x}^\frac{(d+2)(p-1)}2(\R\times \R^d)}^\frac{(d+2)(p-1)}2  \nonumber \\
&     +  C_k \sum\limits_{1 \le j \ne m \le k}   \int\int_{\R\times \R^d} |u_n^j|^{\frac{(d+2)(p-1)}2-1} |u_n^m| \,\mathrm{d}x\mathrm{d}t. \label{eq6.111v3}
\end{align}
We see by \eqref{eq6.93v3} and Lemma \ref{le6.8v3} that
\begin{equation}\label{eq6.112v3}
\sum\limits_{1\le j \le k}\norm{u_n^j}_{L_{t,x}^\frac{(d+2)(p-1)}2(\R\times \R^d)}^\frac{(d+2)(p-1)}2 \lesssim \sum\limits_{1\le j\le j_0}
\norm{ u^j}_{S^1(\R)}^\frac{(d+2)(p-1)}2 + \Big(\sum\limits_{j> j_0} \norm{\varphi^j}_{H^1}^2\Big)^\frac{(d+2)(p-1)}4 < \infty.
\end{equation}
Next, we consider the second term on the right side of \eqref{eq6.111v3}.
By the H\"older inequality, \eqref{eq6.48v3} and \eqref{eq6.93v3}, we have
\begin{align}\label{eq6.113v6}
\begin{split}
  \int\int_{\R\times \R^d} |u_n^j|^{\frac{(d+2)(p-1)}2-1} |u_n^m| \,\mathrm{d}x\mathrm{d}t
\lesssim &   \ \norm{u_n^j u_n^m}_{L_{t,x}^\frac{(d+2)(p-1)}4} \norm{u_n^j}_{L_{t,x}^\frac{(d+2)(p-1)}2}^{\frac{(d+2)(p-1)}2 - 2}  \\
\lesssim  &   \ \norm{u_n^j u_n^m}_{L_{t,x}^\frac{(d+2)(p-1)}4}\to 0, \text{ as }  n\to \infty.
\end{split}
\end{align}
Plugging \eqref{eq6.112v3} and \eqref{eq6.113v6} into
\eqref{eq6.111v3}, we obtain that there is $B_1 > 0$ such that
\begin{equation*}
\underset{n\to \infty} \limsup\  \normo{ u_n^{<k}}_{L_{t,x}^\frac{(d+2)(p-1)}2(\R\times \R^d)} \le B_1.
\end{equation*}
Similarly,
we have
\begin{equation*}
 \underset{n\to \infty}\limsup\  \normo{ \jb{\nabla} u_n^{<k}}_{L_{t,x}^\frac{2(d+2)}d(\R\times \R^d)} \le B_1
\end{equation*}
and
\begin{equation*}
\underset{n\to \infty}\limsup\  \normo{ \jb{\nabla} u_n^{<k}}_{L_t^\infty L_x^2(\R\times \R^d)} \le B_1.
\end{equation*}
Thus, we obtain \eqref{eq6.94v3}.
\end{proof}
\begin{lem}[At least one bad profile] \label{le6.10v3}
 Let $j_0$ be as in Lemma \ref{le6.8v3}, then there exists $1 \le j \le j_0$ such that
 \begin{equation*}
 \norm{ u^j}_{L_{t,x}^\frac{2(d+2)}d\cap L_{t,x}^\frac{(d+2)(p-1)}2((T_{min}^j, T_{max}^j) \times \R^d)} = \infty.
 \end{equation*}
\end{lem}
\begin{proof}
 We argue by contradiction.
 Assume
\begin{equation*}
\norm{ u^j}_{L_{t,x}^\frac{2(d+2)}d\cap L_{t,x}^\frac{(d+2)(p-1)}2((T_{min}^j, T_{max}^j) \times \R^d)}< \infty,\ \  \forall\  1 \le j \le j_0.
\end{equation*}
Combining this with Lemma \ref{le6.8v3}, we have
\begin{equation}\label{eq6.126v3}
\norm{ u^j}_{L_{t,x}^\frac{2(d+2)}d\cap L_{t,x}^\frac{(d+2)(p-1)}2((T_{min}^j, T_{max}^j) \times \R^d)} < \infty,\ \ \forall\  j \ge 1.
\end{equation}
This together with Lemma \ref{le6.9v3} implies $u_n^j$ exists globally in time for $j\ge 1$ and hence so does $u_n^{<k}(t) + e^{it\Delta} w_n^k$.

 We now verify $u_n^{<k}(t) + e^{it\Delta} w_n^k$ is an approximate solution to $u_n$ when $n$ and $k$ large enough, then we can use the long time perturbation theory to
give a contradiction.

We see from Lemma \ref{le6.9v3} and \eqref{eq6.57v3} that there exists $A_1, B > 0$ such that
\begin{align}
\underset{n\to \infty }\limsup  \ \norm{\jb{\nabla} (u_n^{<k}(t) + e^{it\Delta} w_n^k)}_{L_t^\infty L_x^2(\R\times \R^d)} \le A_1,\label{eq5.36v12}\\
\underset{n\to \infty }\limsup  \ \norm{u_n^{<k}(t) + e^{it\Delta} w_n^k}_{L_{t,x}^\frac{(d+2)(p-1)}2 \cap L_{t,x}^\frac{2(d+2)}d(\R\times \R^d)} \le B.\label{eq6.128v3}
\end{align}
Moreover, it follows from \eqref{eq6.51v3} with $t=0$ and
(\ref{eq6.65v3}) that
\begin{align}
& \normo{u_n(0) - u_n^{<k}(0) -  w_n^k}_{H^1} \nonumber\\
= & \normo{\sum\limits_{j=1}^k  e^{i\theta_n^j} e^{-it_n^j\Delta} \varphi^j(x-x_n^j) - \sum\limits_{j=1}^k e^{i\theta_n^j} u^j(-t_n^j, x-x_n^j)}_{H^1}\nonumber\\
\le &
 \sum\limits_{j=1}^k \normo{  e^{-i t_n^j \Delta} \varphi^j - u^j(-t_n^j) }_{H^1}\longrightarrow0, \ \ \text{as}\ n\rightarrow\infty. \nonumber
\end{align}
Hence,
\begin{align}
&   \normo{u_n(0) -  u_n^{<k}(0) -  w_n^k}_{H^1} \to 0, \text{ as } n\to \infty.\label{eq6.130v3}
\end{align}
Next, we claim that as $k\to \infty$
\begin{equation} \label{eq6.132v3}
\underset{n\to \infty} \lim\ \normo{\jb{\nabla} \big[(i\partial_t+
\Delta ) ( u_n^{<k}(t) + e^{it\Delta} w_n^k) -
\mathcal{N}(u_n^{<k}(t) + e^{it\Delta} w_n^k)
\big]}_{L_{t,x}^\frac{2(d+2)}{d+4}} \to 0,
\end{equation}
where
$\mathcal{N}(u) =  |u|^\frac4d u - |u|^{p-1} u.$

   Before proving this claim, we remark that \eqref{eq6.132v3} together with the long-time perturbation theory leads to a contradiction.
    Indeed, by \eqref{eq5.36v12},
 \eqref{eq6.128v3}, \eqref{eq6.130v3} and \eqref{eq6.132v3}, we conclude as a consequence of
 Proposition \ref{pr4.3} that
\begin{equation*}
\norm{u_n}_{L_{t,x}^\frac{2(d+2)}d \cap L_{t,x}^\frac{(d+2)(p-1)}2 (\R\times \R^d)} < \infty
\end{equation*}
when $n$ large enough, which contradicts \eqref{eq5.5}. Hence, Lemma \ref{le6.10v3} holds.

It remains to prove the claim \eqref{eq6.132v3}. Note that
\begin{align*}
& (i\partial_t+ \Delta)  (u_n^{<k}(t) + e^{it\Delta} w_n^k) -\mathcal{N}(u_n^{<k}(t) + e^{it\Delta} w_n^k) \\
= &\sum\limits_{j={1}}^k (i\partial_t + \Delta ) u_n^j -
\mathcal{N}(u_n^{<k}(t)) - \mathcal{N}(u_n^{<k}(t) + e^{it\Delta}
w_n^k) +  \mathcal{N}(u_n^{<k}(t)).
\end{align*}
Hence, it suffices  to show that
\begin{align}\label{eq6.140v3}
\underset{n\to \infty} \lim \normo{\jb{\nabla}\bigg(
\sum\limits_{j=1}^k (i\partial_t + \Delta ) u_n^j
 - \mathcal{N}(u_n^{<k}(t))\bigg)}_{L_{t,x}^\frac{2(d+2)}{d+4}} =
 0,\quad \forall~k\in\N,
 \end{align}
 and
 \begin{align}\label{eq6.139v3}
 \underset{n\to \infty} \lim\normo{\jb{\nabla}\big( \mathcal{N}(u_n^{<k}(t) + e^{it\Delta} w_n^k) - \mathcal{N}(u_n^{<k}(t))\big)}_{L_{t,x}^\frac{2(d+2)}{d+4}} \to 0,
\text{ as }  k\to \infty.
\end{align}

First, we show \eqref{eq6.140v3}. Noting that
\begin{align*}
 (i\partial_t + \Delta ) u_n^j = \mathcal{N}(u_n^j),
\end{align*}
and using  \eqref{eq6.48v3} and Lemma \ref{le6.9v3}, we get as $n
\to \infty$
\begin{align*}
  \normo{\jb{\nabla}\Big(\sum\limits_{j=1}^k (i\partial_t + \Delta ) u_n^j - \mathcal{N}( u_n^{<k})\Big)}_{L_{t,x}^\frac{2(d+2)}{d+4}}
=  \normo{\jb{\nabla}\left( \sum\limits_{j=1}^k  \mathcal{N}(u_n^j)
- \mathcal{N}(\sum\limits_{j=1}^k
u_n^j)\right)}_{L_{t,x}^\frac{2(d+2)}{d+4}} \to 0,
\end{align*}
and \eqref{eq6.140v3} follows.

We now turn to \eqref{eq6.139v3}. By the Fundamental Theorem of
Calculus,
\begin{equation}
F(u)-F(v)=(u-v)\int_0^1F_z\big(v+\theta(u-v)\big)d\theta+\overline{(u-v)}\int_0^1F_{\bar
z}\big(v+\theta(u-v)\big)d\theta,
\end{equation}
we obtain
\begin{align}
& \normo{\jb{\nabla} \big( \mathcal{N}(u_n^{<k}+ e^{it\Delta} w_n^k) - \mathcal{N}(u_n^{<k})\big)}_{L_{t,x}^\frac{2(d+2)}{d+4}}\nonumber \\
\lesssim  & \normo{|u_n^{<k}|^{p-1} e^{it\Delta} w_n^k}_{L_{t,x}^\frac{2(d+2)}{d+4}} + \normo{|u_n^{<k}|^\frac4d e^{it\Delta} w_n^k}_{L_{t,x}^\frac{2(d+2)}{d+4}} \label{eq2v3}\\
&+   \normo{e^{it\Delta} w_n^k}_{L_{t,x}^\frac{2(d+2)p}{d+4}}^p   + \normo{e^{it\Delta} w_n^k}_{L_{t,x}^\frac{2(d+2)}{d}}^\frac{d+4}d \label{eq4v3}\\
&+  \normo{|u_n^{<k}|^{p-1} \nabla e^{it\Delta} w_n^k}_{L_{t,x}^\frac{2(d+2)}{d+4}}
+ \normo{|u_n^{<k}|^\frac4d \nabla e^{it\Delta} w_n^k}_{L_{t,x}^\frac{2(d+2)}{d+4}}\label{eq6v3}\\
&+  \normo{|e^{it\Delta} w_n^k|^{p-1} \nabla
u_n^{<k}}_{L_{t,x}^\frac{2(d+2)}{d+4}}
 + \normo{|e^{it\Delta} w_n^k|^\frac4d \nabla u_n^{<k}}_{L_{t,x}^\frac{2(d+2)}{d+4}}\label{eq8v3}\\
&+   \normo{|e^{it\Delta} w_n^k|^{p-1} \nabla e^{it\Delta} w_n^k}_{L_{t,x}^\frac{2(d+2)}{d+4}} + \normo{|e^{it\Delta} w_n^k|^\frac4d \nabla e^{it\Delta} w_n^k}_{L_{t,x}^\frac{2(d+2)}{d+4}}\label{eq10v3}\\
&+   \normo{|u_n^{<k}|^{p-2} e^{it\Delta} w_n^k \nabla
u_n^{<k}}_{L_{t,x}^\frac{2(d+2)}{d+4}} +
\normo{|u_n^{<k}|^\frac{4-d}d  e^{it\Delta} w_n^k \nabla
u_n^{<k}}_{L_{t,x}^\frac{2(d+2)}{d+4}}. \label{eq12v3}
\end{align}
For the terms \eqref{eq4v3}, \eqref{eq10v3}. By the H\"older
inequality and \eqref{eq6.50v3}, we have
\begin{align*}
\eqref{eq4v3}  & \lesssim \normo{|\nabla|^{\frac{d}2- \frac{d+4}{2p}} e^{it\Delta} w_n^k}_{L_t^\frac{2(d+2)p}{d+4} L_x^\frac{2d(d+2)p}{d(d+2)p-2(d+4)}}^p + \normo{e^{it\Delta} w_n^k}_{L_{t,x}^\frac{2(d+2)}{d}}^\frac{d+4}d \\
            & \lesssim \normo{ \jb{\nabla} e^{it\Delta} w_n^k}_{L_t^\frac{2(d+2)p}{d+4} L_x^\frac{2d(d+2)p}{d(d+2)p-2(d+4)}}^p  +  \normo{e^{it\Delta} w_n^k}_{L_{t,x}^\frac{2(d+2)}{d}}^\frac{d+4}d
           \to 0, \text{ as } n\to \infty, \  k\to \infty.\\
\eqref{eq10v3} & \le \Big[\normo{e^{it\Delta} w_n^k}_{L_{t,x}^\frac{(d+2)(p-1)}2}^{p-1}  + \normo{e^{it\Delta} w_n^k}_{L_{t,x}^\frac{2(d+2)}2}^\frac4d \Big]
                \cdot\normo{\nabla e^{it\Delta} w_n^k}_{L_{t,x}^\frac{2(d+2)}d}\\
               & \lesssim \Big[\normo{|\nabla|^{s_p} e^{it\Delta} w_n^k}_{L_{t}^\frac{(d+2)(p-1)}2  L_x^\frac{2d(d+2)(p-1)}{d(d+2)(p-1) - 8}}^{p-1}
                + \normo{e^{it\Delta} w_n^k}_{L_{t,x}^\frac{2(d+2)}2}^\frac4d\Big]\cdot \normo{\nabla e^{it\Delta} w_n^k}_{L_{t,x}^\frac{2(d+2)}d}\\
              & \  \to 0, \text{ as }  n\to \infty,\  k\to \infty.
\end{align*}
By Lemma \ref{le6.9v3} and \eqref{eq6.50v3}, we also have for \eqref{eq8v3} that
\begin{align*}
& \underset{k\to \infty} \lim \underset{n\to \infty} \lim \Big( \normo{|e^{it\Delta} w_n^k|^{p-1} \nabla u_n^{<k}}_{L_{t,x}^\frac{2(d+2)}{d+4}}
+ \normo{|e^{it\Delta} w_n^k|^\frac4d  \nabla u_n^{<k}}_{L_{t,x}^\frac{2(d+2)}{d+4}}\Big)\\
\le &\  \underset{k\to \infty} \lim \underset{n\to \infty} \lim
\Big[ \normo{e^{it\Delta} w_n^k}_{L_{t,x}^\frac{(d+2)(p-1)}2}^{p-1}
 +    \normo{e^{it\Delta} w_n^k}_{L_{t,x}^\frac{2(d+2)}d}^\frac4d \Big]\cdot\normo{\jb{\nabla} u_n^{<k}}_{L_{t,x}^\frac{2(d+2)}d}   = 0.
\end{align*}
We now consider the terms of the form
\begin{equation*}
\normo{| u_n^{<k}|^q |\nabla|^s e^{it\Delta}
w_n^k}_{L_{t,x}^\frac{2(d+2)}{d+4}}, \ 1 + \frac4d \le q \le 1 +
\frac4{d-2},\  s = 0, 1,
\end{equation*}
which corresponds to \eqref{eq2v3}, \eqref{eq6v3}.

By the H\"older inequality, \eqref{eq6.50v3}, \eqref{eq6.57v3}, \eqref{eq6.94v3}, we have
\begin{align*}
  \normo{\big|u_n^{<k} \big|^{q-1} |\nabla|^s e^{it\Delta} w_n^k}_{L_{t,x}^\frac{2(d+2)}{d+4}}
\lesssim &   \normo{\big| u_n^{<k}\big|^{q-1} |\nabla|^s e^{it\Delta} w_n^k}_{L_{t,x}^\frac{2(d+2)}{d+4}}\\
\lesssim &  \normo{  u_n^{<k}}_{L_{t,x}^\frac{(d+2)(q-1)}2}^{q-1} \normo{|\nabla|^s e^{it\Delta} w_n^k}_{L_{t,x}^\frac{2(d+2)}d}\\
\lesssim  & \normo{  u_n^{<k}}_{L_{t,x}^\frac{(d+2)(q-1)}2}^{q-1}
\normo{|\nabla|^s e^{it\Delta} w_n^k}_{L_{t,x}^\frac{2(d+2)}d}\\
 \to & \ 0, \text{ as } n\to \infty, \ k\to \infty.
 \end{align*}
Thus, we obtain for $1 + \frac4d \le q \le 1+ \frac4{d-2}$, $s = 0,\, 1$,
\begin{align*}
\underset{n\to \infty} \lim \normo{|  u_n^{<k}|^{q-1} |\nabla|^s
e^{it\Delta} w_n^k}_{L_{t,x}^\frac{2(d+2)}{d+4}} \to 0, \text{ as }
k\to \infty.
\end{align*}
We now estimate \eqref{eq12v3}. For $q >2$, we have
\begin{align*}
& \normo{|  u_n^{<k}|^{q-2} e^{it\Delta} w_n^k \nabla u_n^{<k}}_{L_{t,x}^\frac{2(d+2)}{d+4}}\\
\le &  \normo{e^{it\Delta} w_n^k}_{L_{t,x}^\frac{(d+2)(q-1)}2} \normo{\nabla  u_n^{<k}}_{L_{t,x}^\frac{2(d+2)}d} \normo{ u_n^{<k}}_{L_{t,x}^\frac{(d+2)(q-1)}2}^{q-2}\\
\lesssim  & \normo{|\nabla|^{s_q} e^{it\Delta}
w_n^k}_{L_{t}^\frac{(d+2)(q-1)}2
L_x^\frac{2(d+2)d(q-1)}{d(d+2)(q-1)-8}} \normo{\nabla
u_n^{<k})}_{L_{t,x}^\frac{2(d+2)}d} \normo{\nabla
u_n^{<k}}_{L_{t,x}^\frac{(d+2)(q-1)}2}^{q-2} \to 0,
\end{align*}
as $n \to \infty, \ k\to \infty$. Thus \eqref{eq6.139v3} follows.
\end{proof}
We can now show the main result in this section:
\begin{prop}[Existence of the critical element] \label{pr5.8}
Suppose $m_\om^* < m_\om$, then there exists a global solution $u_c\in C(\R, H^1(\R^d))$ to \eqref{eq1.1} such that
\begin{align}
& u_c(t) \in A_{\omega, +}  \text{ and } \cS_\om(u_c(t)) =  m_\om^*,\text{    for }  t\in \R, \label{eq6.160v3}\\
& \|u_c\|_{L_{t,x}^\frac{(d+2)(p-1)}2 \cap L_{t,x}^\frac{2(d+2)}d([0,\infty)\times \R^d)} = \|u_c\|_{L_{t,x}^\frac{(d+2)(p-1)}2 \cap L_{t,x}^\frac{2(d+2)}d((-\infty,0]\times \R^d)} = \infty. \label{eq6.161v3}
\end{align}
\end{prop}
\begin{proof}
By Lemma \ref{le6.8v3}, Lemma \ref{le6.10v3} and reordering indices, there exists $J\le j_0$ that
\begin{equation}\label{eq6.164v3}
\begin{cases}
\norm{u^j}_{L_{t,x}^\frac{(d+2)(p-1)}2 \cap L_{t,x}^\frac{2(d+2)}d((T_{min}^j, T_{max}^j)\times \R^d)} = \infty,  & \text{ for } 1\le j \le J,\\
\norm{ u^j}_{L_{t,x}^\frac{(d+2)(p-1)}2 \cap L_{t,x}^\frac{2(d+2)}d((T_{min}^j, T_{max}^j)\times \R^d)} < \infty,  & \text{ for }  j \ge J.
\end{cases}
\end{equation}
From \eqref{eq6.52v3}, \eqref{eq6.54v3} and \eqref{eq6.55v3},
we have
\begin{align}
& \normo{\jb{\nabla} u_n(0)}_{L^2}^2 - \sum\limits_{j=1}^J \normo{\jb{\nabla} \varphi^j}_{L^2}^2 - \normo{\jb{\nabla} w_n^J}^2 \to 0, 
\nonumber \\
& \cS_\omega(u_n(0)) - \sum\limits_{j=1}^J \cS_\omega(e^{-it_n^j \Delta} \varphi^j) - \cS_\omega(w_n^J) \to 0, \label{eq6.169v3} \\
& \cH_\omega(u_n(0)) - \sum\limits_{j=1}^J \cH_\omega(e^{-it_n^j \Delta} \varphi^j) - \cH_\omega(w_n^J) \to 0,\label{eq6.170v3}
 \text{ as  }   n \to \infty.
\end{align}
 Since $K(u_n(0)) \ge 0$, we have $\cH_\omega(u_n(0)) \le \cS_\omega(u_n(0))$ by \eqref{eq2.3new}. It follows from \eqref{eq6.170v3} and \eqref{eq5.4} that
\begin{equation*}
\cH_\omega(e^{-it_n^j \Delta} \varphi^j),\  \cH_\omega(w_n^J) \le
\cS_\omega(u_n(0)) < \tfrac{m_\omega + m_\omega^*}2,
\end{equation*}
for $1\le j \le J$ and $n$ large enough.

   By \eqref{eq2.4}, we have
\begin{equation}\label{eq6.172v3}
\K(e^{-it_n^j \Delta} \varphi^j) > 0,\  \K(w_n^J) > 0, \text{  for } 1\le j \le J \text{ and $n$ large enough,}
\end{equation}
which together with Lemma \ref{le2.6} shows
\begin{align}\label{eq4.158v6}
   \cS_\omega(e^{-it_n^j \Delta} \varphi^j) \ge  0, \ \cS_\omega( w_n^J ) \ge 0,
\end{align}
for $ 1 \le j \le J$ and $n$ large enough.

 We shall show $J = 1$. Assume for a contradiction that $J\ge 2$. Then, it follows from \eqref{eq5.4} and \eqref{eq6.169v3} that
\begin{equation*}
\underset{n\to \infty} \limsup \ \cS_\omega(e^{-it_n^j \Delta} \varphi^j) < m_\omega^*, \ \forall \,1 \le j \le J,
\end{equation*}
which together with \eqref{eq6.65v3} shows
\begin{equation*} \label{eq6.175v3}
\cS_\omega(u^j(t)) < m_\omega^*, \ \forall\, 1 \le j \le J, \ t \in I^j.
\end{equation*}
Since $u^j$ is a solution to \eqref{eq1.1}, it follows from the definition of $m_\omega^*$ that
\begin{equation*}
\norm{u^j}_{L_{t,x}^\frac{(d+2)(p-1)}2 \cap L_{t,x}^\frac{2(d+2)}d((T_{min}^j, T_{max}^j) \times \R^d)} < \infty, \ \forall \,1 \le j \le J.
\end{equation*}
This contradicts \eqref{eq6.164v3}.
Thus, we have $J = 1$.

  Since $\norm{u^1}_{L_{t,x}^\frac{(d+2)(p-1)}2 \cap L_{t,x}^\frac{2(d+2)}d ( (T_{min}^1, T_{max}^1) \times \R^d)} = \infty$, we have
\begin{equation}\label{eq6.177v3}
\cS_\omega(u^1(t)) \ge m_\omega^*, \text{ for }  t \in (T_{min}^1, T_{max}^1) .
\end{equation}
On the other hand,
by \eqref{eq6.65v3}, \eqref{eq6.169v3}, \eqref{eq4.158v6}, we get
\begin{equation*}
\cS_\omega(u^1(t)) \le m_\omega^*, \text{ for }  t \in (T_{min}^1,
T_{max}^1).
\end{equation*}
Combining this with \eqref{eq6.177v3}, we obtain
\begin{equation}\label{eq6.179v3}
\cS_\omega(u^1(t)) = m_\omega^*, \text{ for }  t \in (T_{min}^1, T_{max}^1) .
\end{equation}
By \eqref{eq6.65v3}, we have
\begin{equation*}
\cS_\omega(u^1(t)) = \underset{n\to \infty} \lim \cS_\omega(e^{-it_n^1 \Delta} \varphi^1),
\end{equation*}
this together with \eqref{eq5.4}, \eqref{eq6.169v3} and \eqref{eq6.179v3} shows
\begin{equation}\label{eq6.181v3}
\cS_\omega(w_n^1) \to 0, \text{ as } n \to \infty.
\end{equation}
Hence, Lemma \ref{le2.6} together with \eqref{eq6.172v3} and \eqref{eq6.181v3} shows
\begin{equation}\label{eq6.182v3}
\norm{w_n^1}_{H^1} \to 0, \text{ as }  n \to \infty.
\end{equation}
We see from \eqref{eq6.51v3}, \eqref{eq6.182v3} that
\begin{equation}\label{eq6.183v3}
\norm{u_n(0,x) - e^{-it_n^1 \Delta} \varphi^1(x-x_n^j)}_{H^1} \to 0, \text{ as }  n \to \infty.
\end{equation}
Now, we shall show $T_{min}^1 = -\infty,\  T_{max}^1 = \infty$.
Assume for a contradiction that $T_{max}^1 < \infty$. Let $\{t_n\}$
be a sequence in $(T_{min}^1, T_{max}^1) $ such that $t_n \nearrow
T_{max}^1$ and put
\begin{align*}
\tilde{u}_n(t) = u^1(t+t_n) \quad \text{  and  } \tilde{I}_n =(T_{min}^1-t_n, T_{max}^1-t_n).
\end{align*}
We see that $\{\tilde{u}_n\}$ satisfies
\begin{align*}
& \tilde{u}_n(t) \in A_{\omega, + } \text{  for }  t \in \tilde{I}_n, \text{ and }  \cS_\omega(\tilde{u}_n) =  m_\omega^*,\\
& \norm{\tilde{u}_n}_{L_{t,x}^\frac{(d+2)(p-1)}2 \cap L_{t,x}^\frac{2(d+2)}d(\tilde{I}_n\times \R^d)} = \infty.
\end{align*}
Then, we can apply the above argument as deriving \eqref{eq6.183v3} to this sequence and find
 that there exists a non-trivial $\psi \in H^1(\R^d)$, a sequence $\{\tau_n\}$ with $\tau_\infty = \underset{n\to \infty} \lim \tau_n \in [-\infty, \infty]$,
 $y_n\in \R^d$ such that
\begin{align*}
\underset{n\to \infty} \lim\norm{u^1(t_n,x) - e^{-i\tau_n \Delta} \psi(x-y_n)}_{H^1}
= \underset{n\to \infty} \lim \norm{\tilde{u}_n(0,x)- e^{-i\tau_n \Delta} \psi(x-y_n)}_{H^1}
= 0.
\end{align*}
This together with the Strichartz estimate yields
\begin{equation}\label{eq6.188v3}
\normo{\jb{\nabla} ( e^{it\Delta} u^1(t_n,x) - e^{i(t-\tau_n)\Delta} \psi(x-y_n)) }_{L_{t,x}^\frac{2(d+2)}d(\R\times \R^d)} \to 0, \text{  as  } n \to \infty.
\end{equation}

{\bf Case 1.}
$\tau_\infty  = \pm \infty$. By the dispersive estimate for the free solution, for any compact interval $I$, we have
\begin{equation*}
\normo{\jb{\nabla} e^{i(t-\tau_n)\Delta} \psi}_{L_{t,x}^\frac{2(d+2)}d(I\times \R^d)} \to 0, \text{ as  } n \to \infty,
\end{equation*}
this together with \eqref{eq6.188v3} yields
\begin{equation}\label{eq6.190v3}
\normo{\jb{\nabla} e^{it\Delta} u^1(t_n)}_{L_{t,x}^\frac{2(d+2)}d(I\times \R^d)} \to 0, \text{  as } n \to \infty.
\end{equation}

{\bf Case 2.}  $\tau_\infty \in \R$. For any interval $I$ with
$\tau_\infty\in I$ and $|I|\ll 1$, we have by \eqref{eq6.188v3},
\begin{align}\label{eq6.191v3}
\underset{n\to \infty} \lim \normo{\jb{\nabla} e^{it\Delta} u^1(t_n)}_{L_{t,x}^\frac{2(d+2)}d(I\times \R^d)}
=   \normo{\jb{\nabla} e^{i(t-\tau_\infty)\Delta} \psi}_{L_{t,x}^\frac{2(d+2)}d(I\times \R^d)}\ll 1.
\end{align}
Then, Proposition \ref{th4.1} together with \eqref{eq6.190v3},
\eqref{eq6.191v3} implies that $u^1$ exists beyond $T_{max}^1$,
which is a contradiction. Thus, $T_{max}^1 = \infty$. Similarly, we
have $T_{min}^1 = -\infty$.

Therefore, $u^1$ is a global solution and it is just the desired
critical element $u_c$ satisfying \eqref{eq6.160v3} and
\eqref{eq6.161v3}.
\end{proof}
We now show the trajectory of the critical element is precompact in the energy space $H^1(\R^d)$ modulo spatial translations.
\begin{prop}[Compactness of the critical element]\label{pr5.9}
  Let $u_c$ be the critical element in Proposition \ref{pr5.8}, then there exists $x(t): \R \to \R^d$ such that $\{ u_c(t,x-x(t)): t\in \R\}$ is precompact in $H^1(\R^d)$.
\end{prop}
\begin{proof}
For $\{t_n\} \subset \R$,
if $t_n \to t^* \in \R,\  \text{ as } n\to \infty$, then we see by the continuity of $u_c(t)$ in $t$ that
\begin{equation*}
u_c(t_n) \to u_c(t^*) \ \text{ in }  H^1(\R^d), \text{  as  } n \to \infty.
\end{equation*}
 If $t_n \to \infty$. Applying the above argument as deriving \eqref{eq6.183v3} to $u_c(t + t_n)$, there exist $(t_n', x_n')\in \R\times \R^d$ and $\phi\in H^1(\R^d)$ that
\begin{equation*}
u_c(t_n,x) - e^{-it_n' \Delta} \phi(x - x_n') \to 0 \ \text{  in }  H^1(\R^d), \text{ as } n\to \infty.
\end{equation*}
{\rm (i)} If $t_n' \to - \infty$, then we have
\begin{align*}
\norm{\jb{\nabla} e^{it\Delta} u_c(t_n)}_{L_{t,x}^\frac{2(d+2)}d([0,\infty)\times \R^d)}
= \norm{\jb{\nabla} e^{it\Delta} \phi}_{L_{t,x}^\frac{2(d+2)}d([-t_n',\infty) \times \R^d)} + o_n(1)
\to 0, \text{ as  } n \to \infty.
\end{align*}
Hence, we can solve \eqref{eq1.1} for $t > t_n$ globally by iteration with small Strichartz norm when $n$ large enough , which
contradicts
\begin{equation*}
\norm{u_c}_{L_{t,x}^\frac{2(d+2)}d \cap L_{t,x}^\frac{(d+2)(p-1)}2([0,\infty)\times \R^d)} = \infty.
\end{equation*}
{\rm (ii)} If $t_n' \to \infty$, then we have
\begin{align*}
\norm{\jb{\nabla} e^{it\Delta} u_c(t_n)}_{L_{t,x}^\frac{2(d+2)}d((-\infty,0]\times \R^d)}%
= \norm{\jb{\nabla} e^{it\Delta} \phi}_{L_{t,x}^\frac{2(d+2)}d((-\infty,-t_n'] \times \R^d)} + o_n(1)
\to 0, \text{ as  } n \to \infty.
\end{align*}
 Hence, $u_c$ can solve \eqref{eq1.1} for $t < t_n$ when $n$ large enough with diminishing Strichartz norm, which contradicts
\begin{equation*}
\norm{u_c}_{L_{t,x}^\frac{2(d+2)}d \cap L_{t,x}^\frac{(d+2)(p-1)}2((-\infty,0])\times \R^d)} = \infty.
\end{equation*}
  Thus $t_n'$ is bounded, which implies that $t_n'$ is precompact, so is $u_c(t_n, x+ x_n')$ in $H^1(\R^d)$.

Similar argument makes sense when $t_n \to -\infty$, we will omit the proof.
\end{proof}
We define for $ R > 0$, $x_0\in \R^d$,
\begin{equation*}
\tilde{\E}_{R,x_0}(u(t)) = \int_{|x - x_0|\ge R} |\nabla u(t)|^2 + |u(t)|^\frac{2(d+2)}d + |u(t)|^{p+1} \,\mathrm{d}x, 
\end{equation*}
then by the compactness of the critical element, we have
\begin{cor}\label{co5.11}
 Let $u_c$ be the critical element in Proposition \ref{pr5.8}, then for any $ \epsilon > 0$, there exist $R_0(\epsilon) > 0$ and $x(t): \R \to \R^d$ such that
\begin{equation*}
\tilde{\E}_{R_0, x(t)}(u_c(t)) \le \epsilon \E(u_c), \ \forall \,t \in \R.
\end{equation*}
\end{cor}
\begin{rem}
In particular, for the radial data $u_0\in H^1$, by the same argument
as in \cite{Miao-Xu-Zhao1}, we have $x(t)\equiv 0$, i.e.
$$\tilde{\E}_{R_0, 0}(u_c(t)) \le \epsilon \E(u_c), \ \forall \,t \in \R.$$
\end{rem}

\section{Extinction of the critical element}\label{se5}
In this section, we prove the non-existence of the critical element by deriving a contradiction from Proposition \ref{pr5.9} and the Virial identity in the radial case.

For a bounded real function $\phi\in C^\infty(\R^d)$, we can define the virial quantity:
\begin{equation}\label{eq6.1v10}
V_R(t) = \int_{\R^d} \phi_R(x)  |u(t,x)|^2  \,\mathrm{d}x, \text{ where }  \phi_R(x) = R^2 \phi\Big(\frac{|x|}R\Big), \,\forall \, R > 0.
\end{equation}
Then, for $u \in C(I; H^1(\R^d))$, we have
\begin{align}
V_R'(t)  = &  2R\cdot \, \Im \int_{\R^d} \phi'\Big(\frac{|x|}R\Big) \frac{x}{|x|}  \cdot \nabla u(t,x) \,\overline{u(t,x)} \,\mathrm{d}x, \label{eq6.2v10}\\
V_R''(t)  = &  4\Re \int \partial_j\partial_k \phi_R(x)
\overline{\partial_j u(t,x)} \partial_k u(t,x) \,\mathrm{d}x
- \int \Delta^2 \phi_R(x) |u_c(t,x)|^2 \,\mathrm{d}x  \nonumber \\
   &  - \tfrac{2(p-1)}{p+1} \int \Delta \phi_R(x) |u(t,x)|^{p+1} \,\mathrm{d}x
+ \tfrac4{d+2} \int \Delta \phi_R(x) |u(t,x)|^\frac{2(d+2)}d \,\mathrm{d}x. \label{eq6.3v10}
\end{align}
\begin{thm}
There does not exist the radial critical element $u_c$ of \eqref{eq1.1} in Proposition \ref{pr5.8}.
\end{thm}
\begin{proof}
  Let the weight function $\phi$ in \eqref{eq6.1v10} be a smooth, radial function satisfying $0 \le \phi \le 1$, and
\begin{equation*}
\phi(x) =
\begin{cases}
|x|^2,  & \ |x|\le 1,\\
0, & \ |x|\ge 2.
\end{cases}
\end{equation*}
On the one hand,
by \eqref{eq6.2v10}, we have
\begin{equation}\label{eq5.4v6}
|V_R'(t)| \lesssim R, \  \forall\, R >0,\  t\in \R.
\end{equation}

On the other hand, by \eqref{eq6.3v10}, we have
\begin{equation}\label{eq5.5v6}
\begin{split}
V_R''(t)
 = &\   4\int \phi_R''(r) |\nabla u_c(t,x)|^2 \,\mathrm{d}x - \int \Delta^2 \phi_R(x) |u_c(t,x)|^2 \,\mathrm{d}x\\
 & - \tfrac{2(p-1)}{p+1} \int \Delta \phi_R(x) |u_c(t,x)|^{p+1} \,\mathrm{d}x
+ \tfrac4{d+2} \int \Delta \phi_R(x) |u_c(t,x)|^\frac{2(d+2)}d \,\mathrm{d}x\\
 =  &  \  8 \K(u_c) + \tfrac{C}{R^2} \int_{R\le |x|\le 2R} |u_c|^2 \,\mathrm{d}x\\
  &+ C\int_{R\le |x|\le 2R} |\nabla u_c(t)|^2 +  |u_c(t)|^{p+1} + |u_c(t)|^\frac{2(d+2)}d \,\mathrm{d}x.
\end{split}
\end{equation}
By Lemma \ref{le2.9} and Lemma \ref{le2.6}, we have
\begin{align*}
8 \K(u_c)  \ge & \, \min\Big\{\tfrac{d(p-1)-4}{d(p-1)}\Big(\norm{\nabla
u_c(t)}_{L^2}^2 + \tfrac{d}{d+2}
 \norm{u_c(t)}_{L^\frac{2(d+2)}d}^\frac{2(d+2)}d\Big), \delta\Big(m_\omega -
 \cS_\omega(u_c(t))\Big)\Big\}\\
\gtrsim & \, \E(u_c(t)).
\end{align*}
Thus, choosing $\epsilon > 0$ small enough and $R = R(\epsilon)$ large enough, then by Corollary \ref{co5.11} with $x(t)\equiv0$, we get
\begin{equation*}
V_R''(t) \gtrsim \E(u_c(t)) = \E(u_0).
\end{equation*}
This together with \eqref{eq5.4v6} implies for $T > 0$,
\begin{equation*}
T \cdot  \E(u_0)\lesssim \left|\int_0^T V_R''(t)\,\mathrm{d}t\right| = |V_R'(T) - V_R'(0)|  \lesssim R.
\end{equation*}
Taking $T$ large enough, we obtain a contradiction unless $u_c \equiv 0$, which is impossible due to
$\norm{u_c}_{L_{t,x}^\frac{2(d+2)}d \cap L_{t,x}^\frac{(d+2)(p-1)}2(\R \times \R^d)}  = \infty$.
\end{proof}
\section{Blow-up}
  We will show the blow-up result in Theorem \ref{th1.4}.

  Let the weight function $\phi$ in \eqref{eq6.1v10} be a smooth, radial function(\cite{Ogawa-Tsutsumi1}) satisfying
$\phi(r) = r^2$ for $r \le 1$, $ \phi''(r) \le 2$ and $\phi(r)$ is constant for $r \ge 3$.

By \eqref{eq6.3v10}, we have
\begin{equation}\label{eq7.2v7}
\begin{split}
V_R''(t)
 \le & \
4\int 2|\nabla u|^2 - \tfrac{d(p-1)}{p+1} |u|^{p+1} + \tfrac{2d}{d+2} |u|^\frac{2(d+2)}d\,\mathrm{d}x\\
 & + \tfrac{C}{R^2} \int_{R\le |x|\le 3R} |u(t,x)|^2 \,\mathrm{d}x + C \int_{R \le |x|\le 3R} |u|^{p+1} + |u|^\frac{2(d+2)}d \,\mathrm{d}x\\
 = & \  8 \K(u) + \tfrac{C}{R^2} \int_{R\le |x|\le 3R} |u|^2 \,\mathrm{d}x + C\int_{R\le |x|\le 3R} |u(t)|^{p+1} + |u(t)|^\frac{2(d+2)}d \,\mathrm{d}x.
\end{split}
\end{equation}
Since $u$ is radial, we have the following radial Sobolev inequalities
\begin{align*}
\|u\|_{L^{p+1}(|x|\ge R)}^{p+1} & \le \tfrac{C}{R^\frac{(d-1)(p-1)}2} \norm{u}_{L^2(|x|\ge R)}^\frac{p+3}2 \norm{\nabla u}_{L^2(|x|\ge R)}^\frac{p-1}2,\\
\|u\|_{L^\frac{2(d+2)}d (|x|\ge R)}^\frac{2(d+2)}d & \le
\tfrac{C}{R^\frac{2(d-1)}d } \norm{u}_{L^2(|x|\ge R)}^\frac{2(d+1)}d
\norm{\nabla u}_{L^2(|x|\ge R)}^\frac2d,
\end{align*}
this together with \eqref{eq7.2v7}, the mass conservation and Young's inequality shows
$\forall\, \epsilon > 0$, there exists $R$ large enough that
\begin{align}\label{eq7.2v10}
 V_R''(t) & \le 8 \K(u) + \epsilon \|\nabla u(t)\|_{L^2}^2 + \epsilon.
\end{align}
By $\K(u) < 0$, energy, mass conservation and Lemma \ref{le2.7}, we see
\begin{equation*}
\K(u(t)) < - \left(m_\omega - \cS_\omega(u(t))\right),\ \forall\, t \in I_{max},
\end{equation*}
thus
\begin{equation*}
\|\nabla u(t)\|_{L^2}^2 < \tfrac{d(p-1)}{2(p+1)}
\|u(t)\|_{L^{p+1}}^{p+1} - (m_\omega - \cS_\omega(u)).
\end{equation*}
So we have by \eqref{eq7.2v10},
\begin{align*}
V_R''(t) &  \le 8 \K(u) + \epsilon \|\nabla u(t)\|_{L^2}^2 + \epsilon \\
         &  = 16 \cS_\omega(u) - 8\omega \|u\|_{L^2}^2 + \tfrac{16- 4d(p-1)}{p+1} \|u(t)\|_{L^{p+1}}^{p+1} + \epsilon \|\nabla u(t)\|_{L^2}^2 + \epsilon\\
          & <
16 \cS_\omega(u) - 8\omega \|u\|_{L^2}^2 +
\left(\tfrac{16-4d(p-1)}{p+1} +
\tfrac{d(p-1)\epsilon}{2(p+1)}\right) \|u(t)\|_{L^{p+1}}^{p+1}
         - \epsilon ( m_\omega - \cS_\omega(u)) + \epsilon.
\end{align*}
Here we take $\epsilon > 0$ small enough such that
$\tfrac{16-4d(p-1)}{p+1} + \tfrac{d(p-1)\epsilon}{2(p+1)}< 0.$

We also note by $\K(u) < 0$ and Proposition \ref{le2.4},
\begin{equation*}
m_\omega \le \cH_\omega(u(t)) = \tfrac{\omega}2 \|u(t)\|_{L^2}^2 +
\tfrac{d(p-1)-4}{4(p+1)} \|u(t)\|_{L^{p+1}}^{p+1},
\end{equation*}
so
\begin{equation*}
\tfrac{4(p+1)}{d(p-1)-4}  \left(m_\omega - \tfrac{\omega}2
\|u(t)\|_{L^2}^2\right) \le \|u(t)\|_{L^{p+1}}^{p+1}.
\end{equation*}
Thus, we have
\begin{align*}
 V_R''(t)  \le & \  16 \cS_\omega(u) - 8 \omega \|u\|_{L^2}^2 - \epsilon( m_\omega - \cS_\omega(u)) + \epsilon \\
   &  + \left(\tfrac{16-4d(p-1)}{p+1} + \tfrac{d(p-1)\epsilon}{2(p+1)}\right) \tfrac{4(p+1)}{d(p-1)-4} \left(m_\omega - \tfrac{\omega}2 \|u(t)\|_{L^2}^2\right).
\end{align*}
By $\cS_\omega(u)  < m_\omega$ and the energy, mass conservation, we see there exists $\delta_1 > 0$ small enough such that $\cS_\omega(u) \le (1 - \delta_1)m_\omega$.
 Then, we get
\begin{align*}
 V_R''(t)   \le  &\  16(1-\delta_1) m_\omega - 8\omega \|u\|_{L^2}^2 - \epsilon (m_\omega - \cS_\omega(u)) + \epsilon \\
  & +\left( - \tfrac{4d(p-1)-16}{p+1} + \tfrac{d(p-1)\epsilon}{2(p+1)}\right) \tfrac{4(p+1)}{d(p-1)-4} \left( m_\omega - \tfrac{\omega}2 \|u\|_{L^2}^2\right)\\
   = &  - \left(16\delta_1 + \epsilon \delta_1 - \tfrac{4(p+1)}{d(p-1)-4} \tfrac{d(p-1)\epsilon}{2(p+1)}\right) m_\omega
 - \tfrac{d(p-1)\epsilon}{2(p+1)} \cdot \tfrac{4(p+1)}{d(p-1)-4}\cdot  \frac{\omega}2 \|u\|_{L^2}^2 + \epsilon\\
  \le  &
 -\left(16\delta_1 + \epsilon \delta_1 - \tfrac{2d(p-1)\epsilon}{d(p-1)-4} \right) m_\omega + \epsilon.
\end{align*}
We can take $\epsilon > 0$ small enough that $V_R''(t) \le -4\delta_1 m_\omega$, which implies that $u$ must blow up in finite time.
\begin{rem}
The blowup is shown for $p \le 5$, which leads to the restriction of the blowup result to $d\ge 2$. This is a technical restriction. See also \cite{Holmer-Roudenko, Ogawa-Tsutsumi1, Ogawa-Tsutsumi2} for some related discussion.
\end{rem}

\textbf{Acknowledgements} The authors  were  supported by the NSF of China under
grant No. 10901148, No.11171033 and 11231006.

\end{document}